\documentclass[11pt]{amsart}
\usepackage{mathtools,amssymb,amsfonts,amsthm,amscd,amsmath,amsgen,upref,enumerate,nicefrac,pdfpages}

\usepackage[margin=1in]{geometry}

\usepackage[colorinlistoftodos,prependcaption,color=yellow,textsize=tiny]{todonotes}
\theoremstyle{plain}
\newtheorem{cor}{Corollary}

\newtheorem{prop}[cor]{Proposition}

\newtheorem{thm}[cor]{Theorem}
\newtheorem*{thm*}{Theorem}
\theoremstyle{definition}
\newtheorem{definition}[cor]{Definition}
\newtheorem{remark}[cor]{Remark}

\newtheorem{assumption}[cor]{Assumption}
\numberwithin{cor}{section}
\numberwithin{equation}{section}

\DeclareMathOperator{\C}{C}

\newcommand{\E}{\mathbb{E}}

\renewcommand{\d}{\delta}

\renewcommand{\and}{\quad\textrm{ and }\quad}

\renewcommand{\P}{\mathbb{P}}
\renewcommand{\a}{\alpha}
\renewcommand{\b}{\beta}

\renewcommand{\O}{\Omega}

\newcommand{\sgn}{\text{sgn}}

\newcommand{\Ent}{\text{Ent}}

\newcommand{\F}{\mathcal F}

\newcommand{\R}{\mathbb{R}}

\newcommand{\N}{\mathbb{N}}

\newcommand{\norm}[1]{\| #1 \|}

\newcommand{\ve}{\varepsilon}

\newcommand{\abs}[1]{|#1|}
\providecommand{\ud}[1]{\, \mathrm{d} #1}
\providecommand{\dx}{\ud{x}}
\providecommand{\dy}{\ud{y}}
\providecommand{\dxi}{\ud \xi}
\providecommand{\deta}{\ud{\eta}}

\providecommand{\ds}{\ud{s}}

\providecommand{\dd}{\ud}

\def\XXint#1#2#3{{\setbox0=\hbox{$#1{#2#3}{\int}$ }
\vcenter{\hbox{$#2#3$ }}\kern-.6\wd0}}

\usepackage{graphicx}
\usepackage{color}

\author{Benjamin Fehrman}
\address{Louisiana State University, Baton Rouge 70802, Louisiana, USA}
\email{fehrman@math.lsu.edu}

\author{Benjamin Gess}
\address{Max Planck Institute for Mathematics in the Sciences, 04103 Leipzig, Germany \newline \indent Fakult\"at f\"ur Mathematik, Universit\"at Bielefeld, 33615 Bielefeld, Germany}
\email{Benjamin.Gess@mis.mpg.de}


\date{\today}

\makeatother

\begin{document}

\title{Conservative stochastic PDEs on the whole space}
\begin{abstract}
The purpose of this paper is to establish a well-posedness theory for conservative stochastic partial differential equations on the whole space. This class of stochastic PDEs arises in fluctuating hydrodynamics, and includes the Dean–Kawasaki equation with correlated noise. In combination with the analysis of the authors and Heydecker \cite{FehGesHey2024}, the connection between fluctuating hydrodynamics and macroscopic fluctuation theory in the context of the zero range particle process is made rigorous.
\end{abstract}

\maketitle

\markright{Conservative SPDEs on the whole space}

\section{Introduction}

The purpose of this paper is to extend the well-posedness theory of \cite{FG21} to equations of the type
\begin{equation}\label{i_1} \partial_t\rho = \Delta\Phi(\rho)-\nabla\cdot(\nu(\rho)+\sigma(\rho)\circ\dd\xi)\;\;\textrm{in}\;\;\R^d\times(0,\infty),\end{equation}
with initial data with finite relative entropy with respect to a nonzero, constant density $\gamma\in(0,\infty)$ and for probabilistically stationary space-time noise $\xi$ defined in Section~\ref{sec_setting} below.  These techniques extend to the whole space the results of the two authors \cite{FG21} on the torus, and establish a large deviations principle for the solutions.  The LDP makes rigorous the connection between the small-noise large deviations of the solutions to \eqref{i_1} and the large deviations of the zero range particle process on the whole space which, with the results of Heydecker and the two authors \cite{FehGesHey2024}, establishes the connection between the non-equilibrium statistical mechanics theories of fluctuating hydrodynamics and macroscopic fluctuation theory in this context.

A model example is the generalized Dean--Kawasaki equation with correlated noise
\begin{equation}\label{i_2}\partial_t\rho = \Delta\Phi(\rho)-\nabla\cdot(\Phi(\rho)+\Phi^\frac{1}{2}(\rho)\circ\dd\xi),\end{equation}
for $\Phi(\rho)=\rho^m$ for any $m\in[1,\infty)$, which for the case $m=1$ in It\^o-form becomes
\[\partial_t\rho = \Delta\rho-\nabla\cdot(\rho+\sqrt{\rho}\dd\xi)+\frac{\langle\xi\rangle_1}{2}\Delta\log(\rho),\]
for $\langle \xi\rangle_1$ being the spatially constant quadratic variation of $\xi$ at time $t=1$.  This demonstrates the two fundamental difficulties in treating \eqref{i_2}:  first the singular noise coefficients that are only $\nicefrac{1}{2}$-H\"older continuous, and second the potential lack of integrability and regularity for $\log(\rho)$ in regions that $\rho\simeq 0$ takes small values.  These difficulties were first handled by the authors in \cite{FG21} on the torus by developing the notion of a \emph{renormalized kinetic solution} of \eqref{i_1} in Definition~\ref{d_sol} below.

In comparison to \cite{FG21}, there are several new difficulties specific to the full space case.   The solution theory is based on the equation's kinetic form, which is an $L^1(\R^d)$-based theory.  However, the analysis of the initial fluctuations in the zero range process, see, for example, Benois, Kipnis, and Landim \cite{BKL95}, leads to initial data that has finite relative entropy with respect to a constant, nonzero density.  Since such functions are not integrable, and since even the difference of two such functions need not be integrable, it is needed to extend the solution theory of \cite{FG21} to non-integrable data with finite relative entropy.  For this it is necessary to control certain commutators on the full space, which requires an additional spatial cut-off: see, for example, the analysis of \eqref{cutoff_term} below.   The control of the resulting cut-off errors relies on a novel relative entropy dissipation estimate in Proposition~\ref{prop_entropy} below, which itself requires a careful handling of the resulting Burkholder--Davis--Gundy term at spatial infinity by means of Sobolev embeddings.

The aforementioned kinetic solution theory allows to renormalize the solution $\rho$ away from regions where $\rho$ takes large values.  A primary contribution of \cite{FG21} is to develop a theory that allows to also renormalize the solution away from regions where $\rho$ is small, in order to avoid the singularities of the square root and logarithm.  We emphasize that the act of cutting out small values is rather more extreme than cutting out large values, due to the fact that by cutting out small values the solution will not necessarily satisfy the equation on the set $\{\rho=0\}$.  This is a serious potential source of nonuniqueness which we overcome by insisting that solutions satisfy the relative entropy dissipation estimates \eqref{d_s2} and \eqref{d_s3} in Definition~\ref{d_sol}, which are a consequence of the new estimates Proposition~\ref{prop_entropy} and Proposition~\ref{prop_measure} on the whole space.

The first primary result of this paper extends the well-posedness theory of \cite{FG21} to the whole space for initial data with finite relative entropy with respect to a constant density $\gamma\in(0,\infty)$, and establishes a pathwise $L^1$-contraction for initial data that is integrable with respect to this constant density.  The noise $\xi$ is spatially probabilistically stationary with spatial divergence that satisfies  $\langle\nabla\cdot\xi\rangle_1\in(L^1\cap L^\infty)(\R^d)$.  The function spaces $\Ent_{\Phi,\gamma}$ and $L^1_\gamma$ are defined below in \eqref{def_fre} and \eqref{def_lpg}.

\begin{thm}[cf.\ Theorem~\ref{thm_unique}, Proposition~\ref{prop_entropy}, Proposition~\ref{prop_measure}, Theorem~\ref{thm_exist}]  Let $\xi$ satisfy Assumption~\ref{a_random}, let $\Phi$, $\sigma$, and $\nu$ satisfy Assumptions~\ref{assume_1} and \ref{as_compact}, and let $\gamma\in(0,\infty)$.  Then, for every $\F_0$-measurable $\rho_0\in L^1(\O;\Ent_{\Phi,\gamma}(\R^d))$ there exists a stochastic kinetic solution of \eqref{i_1} in the sense of Definition~\ref{d_sol}.

Furthermore, for $\F_0$-measurable $\rho^1_0,\rho^2_0\in L^1(\O;(L^1_\gamma\cap\Ent_{\Phi,\gamma})(\R^d))$, if $\rho^1,\rho^2$ are stochastic kinetic solutions of \eqref{i_1} in the sense of Definition~\ref{d_sol} with initial data $\rho_0^1,\rho_0^2$ respectively, then $\P$-a.s.\
\[\sup_{t\in[0,T]}\norm{\rho^1(\cdot,t)-\rho^2(\cdot,t)}_{L^1(\R^d)}\leq\norm{\rho^1_0-\rho^2_0}_{L^1(\R^d)}.\]
If $\rho_0\in L^1(\O;\Ent_{\Phi,\gamma}(\R^d))$ is $\F_0$-measurable then any two stochastic kinetic solutions $\rho^1$ and $\rho^2$ of \eqref{i_1} in the sense of Definition~\ref{d_sol} satisfy $\P$-a.s.\ that $\rho^1=\rho^2$ in $\Ent_{\Phi,\gamma}(\R^d)$ for every $t\in[0,T]$.  \end{thm}

Concerning the assumptions on the noise, after writing \eqref{i_1} in its It\^o-form, we obtain
\begin{equation}\label{ns_eq}\partial_t\rho  = \Delta\Phi(\rho)-\nabla\cdot(\nu(\rho)+\sigma(\rho)\dd\xi)+\frac{1}{2}\nabla\cdot\big(\sigma'(\rho)^2\langle\xi\rangle_1\nabla\rho\big)+\frac{1}{4}\nabla\cdot(\sigma(\rho)\sigma'(\rho)\nabla\langle\xi\rangle_1\big),\end{equation}
for the quadratic variation $\langle\xi\rangle_1$ at time $t=1$.   For simplicity, in this work, we restrict to spatially probabilistically stationary noise, for which \eqref{ns_eq} reduces to
\[\partial_t\rho  = \Delta\Phi(\rho)-\nabla\cdot(\nu(\rho)+\sigma(\rho)\dd\xi)+\frac{\langle\xi\rangle_1}{2}\nabla\cdot\big(\sigma'(\rho)^2\nabla\rho\big),\]
with spatially constant quadratic variation $\langle \xi\rangle_1$.  The assumption of spatial probabilistic stationarity is not necessary, and details on the treatment of non-stationary noise can be found in \cite{FG21}.  We furthermore assume that the quadratic variation of the spatial divergence $\langle\nabla\cdot\xi\rangle_1$ is bounded and integrable, which is a property used to treat the cut-off error \eqref{cutoff_term} below at spatial infinity and to close the full space entropy dissipation estimates \eqref{d_s2} and \eqref{d_s3} in Propositions~\ref{prop_entropy} and Propositions~\ref{prop_measure} below.  These are the only two properties of the noise used to prove the well-posedness of the equation.

In order to precisely quantify the scaling in the large deviations principle, we specify a particular class of stationary noise satisfying $\langle\nabla\cdot\xi\rangle_1\in (L^1\cap L^\infty)$.  Precisely, for $\mathbf{a}=(\a,a,A)$ defined by $\a,A\in(0,1)$ and $a=(a_k)_{k\in\N}\in \ell^2(\N)$, we define the noise $\xi^{\mathbf{a}}$ by
\begin{equation}\label{i_eq_noise}\xi^\mathbf{a} = \sqrt{\exp(-A\abs{x}^2)}(\xi*\kappa^\a)+\sqrt{1-\exp(-A\abs{x}^2)}\xi^a,\end{equation}
for a spatial convolution $(\xi*\kappa^\a)$ of a space-time white noise $\xi$, for $\xi^a = \sum_{k=1}^\infty a_k B^k_t$ defined by independent Brownian motions $\{B^k\}_{k\in\N}$ that are independent of $\xi$, and for $\norm{a}_{\ell^2}^2=\norm{\kappa^\a}_{L^2(\R^d)}^2$.  The precise structure of noise \eqref{i_eq_noise} is not essential, and further details can be found in Remark~\ref{remark_a} below.  We then establish a small-noise large deviations principle for the solutions
\begin{equation}\label{i_eq_ldp}\partial_t\rho^\ve = \Delta\Phi(\rho^\ve)-\sqrt{\ve}\nabla\cdot(\sigma(\rho^\ve)\circ\dd\xi^{\mathbf{a}^\ve}),\end{equation}
along certain scaling limits $\ve,\a^\ve,A^\ve\rightarrow 0$ and $\norm{a^\ve}_{\ell^\infty}\rightarrow 0$ with rate function
\begin{equation}\label{i_rate_function}I_{\rho_0}(\rho) = \frac{1}{2}\inf\{\norm{g}^2_{L^2(\R^d\times[0,T])^d}\colon \partial_t\rho = \Delta\Phi(\rho)-\nabla\cdot(\sigma(\rho)g)\;\;\textrm{with}\;\;\rho(0,\cdot)=\rho_0\;\;\textrm{in}\;\;\R^d\times[0,T]\}.\end{equation}
Here the skeleton equation appearing in the rate function is understood in the sense of Definition~\ref{def_skel} below.  The proof is based on the well-posedness of the skeleton equation, which was established on the full space in \cite{FehGesHey2024}, and the weak approach to large deviations first presented in Budhiraja, Dupuis, and Maroulas \cite{BudDupMar2008}, as well as Budhiraja and Dupuis \cite{BudDup2019} and Dupuis and Ellis \cite{DupEll1997}.  In the full space case, the argument is complicated in particular by the relative entropy estimate for the solutions of the controlled SPDE appearing in Proposition~\ref{prop_control} below.  The reason is due to the fact that the noise $\xi^{\mathbf{a}}$ is effectively constant at spatial infinity, and for this reason in the controlled SPDE \eqref{ldp_1} the control $g$ can have an enormous impact at virtually no cost to the energy.  The $\ell^\infty$-norm of $a^\ve$ appears in the scaling regime for the LDP to control this effect.

\begin{thm}[cf. Theorem~\ref{prop_collapse}]  Let $\Phi$ and $\sigma$ satisfy Assumptions~\ref{assume_1},  \ref{as_compact}, and \ref{as_equiv} and let $\xi^{\mathbf{a}^\ve}$ be a sequence defined by \eqref{i_eq_noise} that satisfies $A^\ve,\a^\ve\rightarrow 0$ as $\ve\rightarrow 0$ and that, if $d\geq 2$,
\[\norm{a^\ve}^2_{\ell^\infty} (A^\ve)^{1-\frac{d+2}{2}}\rightarrow 0\;\;\textrm{and}\;\;\ve(\a^\ve)^{-d-2}(A^\ve)^{-\frac{d}{2}}\rightarrow 0,\]
and, if $d=1$,
\[\norm{a^\ve}_{\ell^\infty}\rightarrow 0\;\;\textrm{and}\;\;\ve(\a^\ve)^{-d-2}(A^\ve)^{-\frac{d}{2}}\rightarrow 0.\]
Then, the rate functions $I_{\rho_0}$  defined in \eqref{i_rate_function} are good rate functions with compact level sets on compact sets, and for every $\rho_0\in\Ent_{\Phi,\gamma}(\R^d)$ the solutions $\{\rho^\ve(\rho_0)\}_{\ve\in(0,1)}$ of \eqref{i_eq_ldp} satisfy a large deviations principle with rate function $I_{\rho_0}$ on $L^1([0,T];L^1_{\textrm{loc}}(\R^d))$.  Furthermore, the solutions satisfy a uniform large deviations principle on subsets of $(L^1_\gamma\cap \Ent_{\Phi,\gamma})(\R^d)$ with uniformly bounded entropy with respect to weakly $L^1(\R^d)$-compact subsets.
\end{thm}

\subsection{Comments on the literature}  Stochastic PDEs with conservative noise have been considered among others by Lions, Perthame, and Souganidis \cite{LPS13-2,LPS13,LPS14}, Friz and the second author \cite{FG16}, and the second author and Souganidis \cite{GS14,GS17}.  Most recently, the authors \cite{FehGes21} treated equations like \eqref{i_1} on the torus as well as equations with multiplicative noise, including the nonlinear Dawson--Watanabe equation with correlated noise.  Some earlier works include the second author and Souganidis \cite{GS16-2}, the two authors \cite{FG17}, and Dareiotis and the second author \cite{DG18}, and numerical approaches have been developed by Hoel, Karlsen, Risebro, and Storrosten \cite{HKRS17,HKRS18}, Ba\u{n}as, the second author, and Vieth \cite{BanGesVie2020}, and the second author, Perthame, and Souganidis \cite{GPS16}.  Furthermore, stochastic PDEs of porous media type on unbounded domains have been considered, for example, by Barbu and R\"ockner \cite{BarRoc2018}, Barbu, R\"ockner, and Russo \cite{BarRocRus2015}, the second author, R\"ockner, and Wu \cite{GesRocWu2024}, Kim \cite{Kim2006}, Pardoux \cite{Pard1980}, R\"ockner, Wu, and Xie \cite{RocWuXie2018}, and Ren, R\"ockner, and Wang \cite{RenROcWan2007}.

The solution theory is based on the kinetic formulation of the equation introduced by Lions, Perthame, and Tadmor \cite{LPT94} and Perthame \cite{Per1998}.  See also the contributions of Bendahmane and Karlsen \cite{BenKar2004}, Chen and Perthame \cite{CP03}, De Lellis, Otto, and Westdickenberg \cite{DeLOttWes2003}, and Karlsen and Riseboro \cite{KarRis2003}.

Large deviations for conservative stochastic PDE have been previously considered by Mariani \cite{M10} and Bellettini, Bertini, Mariani, and Novaga \cite{BBMN10}, which include the case of asymptotically vanishing dissipation.  The most closely related works are those of the authors \cite{FehGes19,FG21} and the authors and Dirr \cite{DirFehGes19}.  In the context of singular SPDEs with additive or multiplicative noise, we also mention the works of Cerrai and Freidlin \cite{CF11}, Faris and Jona-Lasinio \cite{FJL82}, Jona-Lasinio and Mitter \cite{JLM90}, and Hairer and Weber \cite{HW15}, and in the context of the stochastic porous media equation of R\"ockner, Wang, and Wu \cite{RocWanWu2006} and Zhang \cite{Zha2020}.

There have recently been several works investigating the use of conservative SPDEs to numerically approximate particle systems.   Cornalba and Fischer \cite{CF23} have shown that a system of independent Brownian motions can be approximated to arbitrary order by a discretization of the Dean--Kawasaki SPDE, and these results were extended by Cornalba, Fischer, Ingmanns, and Raithel \cite{CFIR23} to the case of weakly interacting particles.  See also the related works of Djurdjevac, Kremp, and Perkowski \cite{DKP22} and the second author, Wu, and Zhang \cite{GesWuZha2024}.  We also remark that the inference of such fluctuation corrections from observations of the underlying particle system has been studied by Li, Dirr, Embacher, Zimmer, and Reina \cite{LDEZR2019}.

The weak convergence approach to large deviations has been developed, for example, in Budhiraja and Dupuis \cite{BudDup2019}, Budhiraja, Dupuis, and Maroulas \cite{BudDupMar2008}, and Dupuis and Ellis \cite{DupEll1997}, and it has been used in context of singular SPDE to derive large deviation estimates by Cerrai and Debussche \cite{CD19}.  Further applications include Brze\'zniak, Goldys, and Jegaraj \cite{BGJ17}, Dong, Wu, Zhang, and Zhang \cite{DWZZ20}, and Wu and Zhai \cite{WuZhai2020}.

The Dean--Kawasaki equation was introduced by Dean \cite{D96} and Kawasaki \cite{K94} and has recently been analyzed by Donev, Fai, and Vanden-Eijnden \cite{DFVE14}, Donev and Vanden-Eijnden \cite{DVE14}, Lehmann, Konarovskyi, and von Renesse \cite{LKR19}, and Konarovskyi and von Renesse \cite{KR17,KR15} and in the references therein.  Sturm and von Renesse \cite{StuvRe2009} constructed solutions to modified Dean--Kawasaki equations by means of Dirichlet forms, and a regularized Dean--Kawasaki model was derived and analyzed by Cornalba, Shardlow, and Zimmer \cite{CorShaZim2019,CorShaZim2020}.  An overview of the link between macroscopic fluctuation theory (MFT) and fluctuating hydrodynamics in the context of the Dean--Kawasaki equation can be found in Bouchet, Gaw\c edzki, and Nardini \cite{BGN16}. A comprehensive overview of MFT can be found in Bertini, De Sole, Gabrielli, Jona-Lasinio, and Landim \cite{BDSGJLL15}.

Large deviations of the zero range process on the whole space were analyzed by Benois, Kipnis, and Landim \cite{BKL95}.  We also refer to Kipnis and Landim \cite{KL99}, Evans and Hanney \cite{EH05}, and the references therein for a detailed account of theory.

\subsection{Overview}  The assumptions on the nonlinearity $\Phi$ are presented in each section and we observe that every assumption is satisfied by the model examples $\Phi(\xi)=\xi^m$ and $\sigma(\xi)=\xi^\frac{m}{2}$ for every $m\in[1,\infty)$.  Section~\ref{section_rks} introduces the notion of a renormalized kinetic solution of \eqref{i_1} in Definition~\ref{d_sol} and is broken into three parts.  Section~\ref{sec_setting} defines the space of functions with finite relative entropy with respect to a constant density $\gamma\in(0,\infty)$ and introduces the assumptions on the noise $\xi$ in Definition~\ref{a_random}.  In Section~\ref{subsection_rks} we define a renormalized kinetic solution, in Section~\ref{section_unique} we prove that the solutions are unique, and in Section~\ref{section_exist} we prove that solutions exist.  We establish the large deviations principle in Section~\ref{sec_ldp}.

\section{Renormalized kinetic solutions of \eqref{i_1}}\label{section_rks}

\subsection{The setting and the randomness in the equation}\label{sec_setting}

In this section, we fix once and for all the setting, the integrability of the initial condition, and the randomness in the equation.  Due to the relevance to the large deviations of the zero range process, we will study \eqref{i_1} in spaces of finite relative entropy with respect to a fixed constant density $\gamma\in(0,\infty)$.  For every $\gamma\in(0,\infty)$, let $\Psi_{\Phi,\gamma}$ be defined by
\[\Psi_{\Phi,\gamma}(\gamma)=0\;\;\textrm{and}\;\;\Psi_{\Phi,\gamma}'(\xi) = \log\Big(\frac{\Phi(\xi)}{\Phi(\gamma)}\Big),\]
for which we have that $\Psi_{\Phi,\gamma}(\xi)\geq 0$ for every $\xi\in[0,\infty)$, that $\Psi_{\Phi,\gamma}$ is convex, and that $\Psi_{\Phi,\gamma}(\xi)=0$ if and only if $\xi=\gamma$ whenever $\Phi$ is strictly increasing.  In the porous media case $\Phi(\xi)=\xi^m$,
\[\Psi_{\Phi,\gamma}(\xi) =m\Big(\xi\log\big(\frac{\xi}{\gamma}\big)-(\xi-\gamma)\Big).\]
The space of functions with finite relative entropy is then
\begin{equation}\label{def_fre}\Ent_{\Phi,\gamma}(\R^d) =\big\{\rho\colon\R^d\rightarrow\R\;\;\textrm{nonnegative and measurable with}\;\;\int_{\R^d}\Psi_{\Phi,\gamma}(\rho)<\infty\big\},\end{equation}
which is a complete, separable metric space with respect to the metric
 \[d(f,g) = \int_{\R^d}\abs{\tilde{\Psi}_{\Phi,\gamma}(f)-\tilde{\Psi}_{\Phi,\gamma}(g)},\]
 for $\tilde{\Psi}_{\Phi,\gamma}(\xi) = \Psi_{\Phi,\gamma}(\xi)$ if $\xi\geq \gamma$ and $\tilde{\Psi}_{\Phi,\gamma}(\xi) = -\Psi_{\Phi,\gamma}(\xi)$ if $\xi\in[0,\gamma]$.  The completeness and separability follow from the fact that $\tilde{\Psi}_{\Phi,\gamma}$ is convex and strictly increasing.  We also define the shifted $L^p$-spaces
 \begin{equation}\label{def_lpg} L^p_\gamma(\R^d) = \{f\colon\R^d\rightarrow\R\;\;\textrm{nonnegative and measurable with}\;\;\int_{\R^d}(f-\gamma)^p<\infty\}.\end{equation}
We will now introduce the type of noise entering the equation.
 
Let $(\O,\F,\P)$ be a complete probability space with a right-continuous filtration $(\F_t)_{t\in[0,\infty)}$ and with independent, $d$-dimensional, $\F_t$-adapted Brownian motions $\{B^k,\tilde{B}^k\}_{k\in\N}$.  An $\R^d$-valued space-time white noise on $\R^d$ admits the spectral representation
\[\xi = \sum_{k=1}^\infty f_k(x)B^k_t,\]
for an orthonormal $L^2(\R^d)$-basis $\{f_k\}_{k\in\N}$, for which $\dd \xi = \sum_{k=1}^\infty f_k(x)\dd B^k_t$ is distributionally a space-time white noise.  For every $\a\in(0,1)$ let $\kappa^\a$ be a standard compactly supported convolution kernel of scale $\a$ on $\R^d$, and for every $\a\in(0,1)$ let $\xi^\a$ be defined by
\begin{equation}\label{n_1}\xi^\a= (\xi*\kappa^\a) = \sum_{k=1}^\infty (f_k*\kappa^\a)B^k_t.\end{equation}
This noise is probabilistically stationary in the sense that, for every $(x,t)\in\R^d\times[0,T]$,
\[\langle \xi^\a\rangle_t(x) = t\norm{\kappa^\a}^2_{L^2(\R^d)}.\]
We also introduce an independent, spatially constant noise, for every $a=(a_k)_{k\in\N}\in \ell^2(\N)$,
\begin{equation}\label{n_2}\xi^a = \sum_{k=1}^\infty a_k\tilde{B}^k_t,\end{equation}
for which we have that $\langle \xi^a\rangle_t(x) = t\norm{a}_{\ell^2}^2$.  We will study a spatially probabilistically stationary noise based on the above two constructions.

\begin{assumption}\label{a_random} Let $(\O,\F,\P)$ be a complete probability space equipped with a right-continuous filtration $(\F_t)_{t\in[0,\infty)}$ and independent, $d$-dimensional, $\F_t$-adapted Brownian motions $\{B^k,\tilde{B}^k\}_{k\in\N}$.  Let $\a\in(0,1)$, let $\kappa^\a$ be a standard compactly supported convolution kernel on $\R^d$ of scale $\a$, and let $a\in\ell^2(\N)$ satisfy $\norm{a}^2_{\ell^2}=\norm{\kappa^\a}^2_{L^2(\R^d)}$.  For every $A\in(0,1)$ and $\mathbf{a}=(\a,a,A)$, let $\xi^{\mathbf{a}}$ be the noise
\[\xi^{\mathbf{a}} = \sqrt{\exp(-A\abs{x}^2)}\xi^\a+\sqrt{1-\exp(-A\abs{x}^2)}\xi^a,\]
for $\xi^\a$ and $\xi^a$ defined in \eqref{n_1} and \eqref{n_2} above. \end{assumption}

\begin{remark}\label{remark_a}  There are two essential aspects of the noise defined in Assumption~\ref{a_random} for our arguments.  With no changes to the proofs in Sections~\ref{section_unique} and \ref{section_exist} we could consider arbitrary noise of the form
\[\xi(x,t) = \sum_{k\in\N}f_k(x)B^k_t,\]
provided that the noise is spatially probabilistically stationary in the sense that $\langle\xi\rangle_1 = \sum_{k=1}^\infty f_k^2$ is bounded and constant on $\R^d$, and provided that $\langle \nabla\cdot\xi\rangle_1  = \sum_{k=1}^\infty \abs{\nabla f_k}^2\in (L^1\cap L^\infty)(\R^d)$.  We consider the specific noise $\xi^{\textbf{a}}$ because it allows us to identify precisely the scaling regime for the large deviations principle of Theorem~\ref{prop_collapse} below.  However, while the probabilistic stationarity and exponential cutoff somewhat simplify the structure and analysis of the equation, they are not necessary.   Details on the treatment of non-stationary noise can be found in \cite{FG21}.\end{remark}

\subsection{Renormalized kinetic solutions of \eqref{i_1}.}\label{subsection_rks}

We will rewrite \eqref{i_1} in its It\^o-formulation, for which we have, using Assumption~\ref{a_random} that, for $\mathbf{a}=(\a,a,A)$ and for the spatially constant quadratic variation $\langle\xi^\mathbf{a}\rangle_1$ at time $t=1$,
\begin{equation}\label{k_1}\partial_t\rho = \Delta\Phi(\rho)-\nabla\cdot(\nu(\rho)+\sigma(\rho)\dd\xi^{\mathbf{a}})+\frac{\langle\xi^\mathbf{a}\rangle_1}{2}\nabla\cdot((\sigma'(\rho))^2\nabla\rho).\end{equation}
The derivation of this equation relies on the fact that the noise is probabilistically stationary in the sense of Remark~\ref{remark_a}.  In the case of the generalized Dean--Kawasaki equation
\[\partial_t\rho = \Delta\rho-\nabla\cdot(\sqrt{\rho}\circ\dd\xi^{\mathbf{a}}),\]
equation \eqref{k_1} already illustrates several difficulties.  The first is the irregularity of the noise coefficient, which is only $\nicefrac{1}{2}$-H\"older continuous and appears under the divergence.  The second is the singularity of the Stratonovich-to-It\^o correction, which in this case takes the form $\frac{\langle\xi^\mathbf{a}\rangle_1}{8}\Delta\log(\rho)$.  Since the logarithm of the solution is not known to be $H^1$-regular, we cannot interpret \eqref{k_1} based on its classical weak formulation.  We instead introduce a generalized solution theory based on the equation's kinetic form.

Since a more complete overview of the kinetic formulation can be found in \cite{CP03} and \cite{FG21}, we will only highlight some of the main ideas here.  The derivation is based on studying the equation satisfied by nonlinear functions $S$ of the solution $\rho$.  When $S$ is convex, an application of It\^o's formula and a viscous regularization of the equation lead to the entropy inequality that, for the spatially inhomogenous quadratic variation $\langle \nabla\cdot\xi^\mathbf{a}\rangle_1 $ at time $t=1$, for every nonnegative $\psi\in\C^\infty_c(\R^d)$,
\begin{align*}
& \int_{\R^d}S(\rho_r)\psi\Big|_{r=0}^t \leq  -\int_0^t\int_{\R^d}\psi S''(\rho)\Phi'(\rho)\abs{\nabla\rho}^2-\int_0^t\int_{\R^d}S'(\rho) \Phi'(\rho)\nabla\rho\cdot\nabla\psi
\\ & \quad -\int_0^t\int_{\R^d}\psi S'(\rho)\nabla\cdot(\nu(\rho)+\sigma(\rho)\dd\xi^{\mathbf{a}})-\frac{\langle\xi^\mathbf{a}\rangle_1}{2}\int_0^t\int_{\R^d}S'(\rho)\sigma'(\rho)^2\nabla\rho\cdot\nabla\psi
\\ & \quad +\frac{1}{2}\int_0^t\int_{\R^d}\langle \nabla\cdot\xi^\mathbf{a}\rangle_1 \sigma^2(\rho) S''(\rho)\psi.
\end{align*}
This inequality relies crucially on the nonnegativity of $S''(\rho)\psi$, which explains the convexity requirement on $S$ and the nonnegativity of $\psi$.  The kinetic formulation quantifies the entropy inequality exactly, which allows to consider signed test functions $\psi$ and nonconvex entropies $S$---essential requirements for our solution theory here, since it will be necessary for us to cut out both large and small values of the solution in order to give a meaning to \eqref{k_1}.

The kinetic function $\overline{\chi}$ is defined after introducing an additional velocity variable $\xi\in\R$, for which we have that
\[\overline{\chi}(\xi,s) = \mathbf{1}_{\{0<\xi<s\}}-\mathbf{1}_{\{s<\xi<0\}},\]
and the kinetic function $\chi$ of the solution $\rho$ is defined by
\[\chi(\xi,x,t) = \overline{\chi}(\xi,\rho(x,t)).\]
Based on the distributional equalities
\[\nabla_x\chi(\xi,x,t) = \delta_{\rho(x,t)}\nabla\rho(x,t)\;\;\textrm{and}\;\;\partial_\xi\chi(\xi,x,t) = \delta_0-\delta_{\rho(x,t)}\;\;\textrm{and}\;\;S(\rho)=\int_\R\chi(\xi,x,t) S'(\xi)\dxi,\]
for the one-dimensional Dirac delta distribution $\delta_0$ and $\delta_{\rho(x,t)}=\delta_0(\xi-\rho(x,t))$, we have from the above equation that, for every smooth $S$ satisfying $S'(0)=0$, for $\Psi(\xi,x) = S'(\xi)\psi(x)$,
\begin{align*}
& \int_\R\int_{\R^d}\chi \Psi\Big|_{r=0}^t \leq  -\int_0^t\int_\R\int_{\R^d}\big(\delta_\rho\Phi'(\xi)\abs{\nabla\rho}^2\big)\partial_\xi\Psi-\int_0^t\int_{\R^d}\Phi'(\rho)\nabla\rho\cdot(\nabla_x\Psi)(x,\rho)
\\ & \quad -\int_0^t\int_{\R^d}\nabla_x\cdot(\nu(\rho)+\sigma(\rho)\dd\xi^{\mathbf{a}})\Psi(x,\rho)-\frac{\langle\xi^\mathbf{a}\rangle_1}{2}\int_0^t\int_{\R^d} \sigma'(\rho)^2\nabla_x\rho\cdot(\nabla_x\Psi)(x,\rho)
\\ & \quad +\frac{1}{2}\int_0^t\int_{\R^d} \langle \nabla\cdot\xi^\mathbf{a}\rangle_1 \sigma^2(\rho)(\partial_\xi\Psi)(x,\rho).
\end{align*}
The kinetic formulation quantifies this inequality exactly using a kinetic defect measure, which is a nonnegative Radon measure $q$ on $\R^d\times\R\times[0,T]$ that satisfies, in the sense of measures,
\[\delta_{\rho(x,t)}(\xi)\Phi'(\xi)\abs{\nabla\rho}^2\leq q(\xi,x,t),\]
for which we have, using the density of linear combinations of functions of the type $S'(\xi)\psi(x)$ in $\C^\infty_c(\R^d\times\R)$, for every $\Psi\in\C^\infty_c(\R^d\times\R)$,
\begin{align*}
& \int_\R\int_{\R^d}\chi \Psi\Big|_{r=0}^t =  -\int_0^t\int_\R\int_{\R^d}\partial_\xi\Psi q-\int_0^t\int_{\R^d}\Phi'(\rho)\nabla\rho\cdot(\nabla_x\Psi)(x,\rho)
\\ & \quad -\int_0^t\int_{\R^d}\nabla_x\cdot(\nu(\rho)+\sigma(\rho)\dd\xi^{\mathbf{a}})\Psi(x,\rho)-\frac{\langle\xi^\mathbf{a}\rangle_1}{2}\int_0^t\int_{\R^d} \sigma'(\rho)^2\nabla_x\rho\cdot(\nabla_x\Psi)(x,\rho)
\\ & \quad +\frac{1}{2}\int_0^t\int_{\R^d} \langle \nabla\cdot\xi^\mathbf{a}\rangle_1\sigma^2(\rho)(\partial_\xi\Psi)(x,\rho).
\end{align*}
This equation is the basis for our solution theory, which allows us to consider test functions that are compactly supported in $\R^d$ and in the interval $(0,\infty)$ in the velocity variable.  These choices enforce integrability by localizing the solutions away from infinity and they avoid the singularities of the coefficients at zero.  

We first define the notion of kinetic defect measure and then define a renormalized kinetic solution of \eqref{k_1}.  In what follows, we will often evaluate the derivative of a function $\psi\in\C^\infty_c(\R^d\times\mathbb{R})$ at the point $\xi=\rho(x,t)$.  We will write $\nabla$ for the gradient in the spatial $x$-variable and we will write
\[(\nabla\psi)(\rho(x,t),x)=\nabla\psi(\xi,x)\rvert_{\xi=\rho(x,t)}\]
to mean the gradient $\nabla\psi$ evaluated at the point $(\rho(x,t),x)$. Also, for $\F_t$-adapted processes $g_t\in L^2(\O\times[0,T];L^2(\R^d))$ and $h_t\in L^2(\O\times[0,T];H^1(\R^d))$, we will write
\[\int_0^t\int_{\R^d}g_s\nabla\cdot\Big(h_s\dd\xi^{\mathbf{a}}\Big) =\sum_{k=1}^\infty\Big(\int_0^t\int_{\R^d}g_sf_k\nabla h_s\cdot \dd B^k_s+\int_0^t\int_{\R^d}g_sh_s\nabla f_k\cdot \dd B^k_t\Big),\]
for stochastic integrals interpreted in the It\^o sense.  Finally, the space $L^1_{\textrm{loc}}(\R^d)$ is equipped with the topology of local strong $L^1$-convergence.  That is, a sequence $\rho_n\rightarrow\rho$ in $L^1_{\textrm{loc}}(\R^d)$ if and only if $\rho_n\rightarrow\rho$ strongly in $L^1(B_R)$ for every $R>0$.

\begin{definition}  Under Assumption~\ref{a_random}, a kinetic measure is a map $q$ from $\O$ to the space of nonnegative, locally finite Radon measures on $\R^d\times\R\times[0,T]$ that satisfies the property that, for every $\psi\in\C^\infty_c(\R^d\times\R)$,
\[(\omega,t)\in\O\times[0,T]\rightarrow \int_0^t\int_\mathbb{R}\int_{\R^d}\psi(\xi,x)\dd q(\omega)\;\;\textrm{is $\F_t$-predictable.}\]
\end{definition}

 \begin{definition}\label{d_sol}  Under Assumption~\ref{a_random}, let $\gamma\in(0,\infty)$ and let $\rho_0\in L^1(\O;\Ent_{\Phi,\gamma}(\R^d))$ be $\F_0$-measurable.  A \emph{stochastic kinetic solution} of \eqref{k_1} is a nonnegative, $\P$-a.s.\ continuous $L^1_{\textrm{loc}}(\R^d)$-valued $\F_t$-predictable function $\rho\in L^\infty(\O\times[0,T];\Ent_{\Phi,\gamma}(\R^d))$ that satisfies the following two properties.
\begin{enumerate}[(i)]
\item \emph{Local relative integrability of the fluxes}:  we have that
\[(\sigma(\rho)-\sigma(\gamma))\in L^2(\O;L^2_{\textrm{loc}}(\R^d\times[0,T]))\;\;\textrm{and}\;\;(\nu(\rho)-\nu(\gamma))\in L^1(\O;L^1_{\textrm{loc}}(\R^d\times[0,T])^d).\]
\item \emph{The entropy estimate}: we have that
\begin{equation}\label{d_s2} \sup_{t\in[0,T]}\int_{\R^d}\Psi_{\Phi,\gamma}(\rho)+\int_0^T\int_{\R^d}\abs{\nabla\Phi^\frac{1}{2}(\rho)}^2<\infty.\end{equation}
\end{enumerate}
Furthermore, there exists a kinetic measure $q$ that satisfies the following three properties.
\begin{enumerate}[(i)]
\setcounter{enumi}{3}
\item \emph{Regularity and the entropy estimate}: $\P$-a.s.\ as nonnegative measures on $\R^d\times\R\times[0,T]$,
\begin{equation}\label{d_s3}\delta_{\rho(x,t)}\Phi'(\xi)\abs{\nabla\rho}^2\leq q\;\;\textrm{and}\;\;\int_0^T\int_0^\infty\int_{\R^d}\frac{1}{\xi}q<\infty.\end{equation}
\item \emph{Vanishing at infinity}:  we have $\P$-a.s. that
\[\liminf_{M\rightarrow\infty}\big(q(\R^d\times[M,M+1]\times[0,T])\big)=0.\]
\item \emph{The equation}: for every $\psi\in \C^\infty_c(\R^d\times(0,\infty))$, $\P$-a.s.\ for every $t\in[0,T]$,
\begin{align}\label{2_5000}
& \int_\mathbb{R}\int_{\R^d}\chi(\xi,x,r)\psi(\xi,x) = \int_\mathbb{R}\int_{\R^d}\overline{\chi}(\rho_0)\psi(\xi,x)-\int_0^t\int_\R\int_{\R^d}\partial_\xi\psi q
\\ \nonumber & \quad -\int_0^t\int_{\R^d}\Phi'(\rho)\nabla\rho\cdot(\nabla_x\psi)(x,\rho)-\int_0^t\int_{\R^d}\nabla\cdot(\nu(\rho)+\sigma(\rho)\dd\xi^{\mathbf{a}})\psi(x,\rho)
\\ \nonumber & \quad -\frac{\langle\xi^\mathbf{a}\rangle_1}{2}\int_0^t\int_{\R^d} \sigma'(\rho)^2\nabla\rho\cdot(\nabla_x\psi)(x,\rho)+\frac{1}{2}\int_0^t\int_{\R^d} \langle \nabla\cdot\xi^\mathbf{a}\rangle_1\sigma^2(\rho)(\partial_\xi\psi)(x,\rho).
\end{align}
\end{enumerate}
\end{definition}

\subsection{Uniqueness of renormalized kinetic solutions}\label{section_unique}  We will establish the uniqueness of solutions for coefficients satisfying the following assumption.  Furthermore, we show that solutions with initial data in $(L^1_\gamma\cap\Ent_{\Phi,\gamma})(\R^d)$ satisfy a pathwise $L^1$-contraction.

\begin{assumption}\label{assume_1}  Assume that $\Phi,\sigma\in\C([0,\infty))\cap \C^{1,1}_{\textrm{loc}}((0,\infty))$ and $\nu\in\C([0,\infty))^d\cap \C^1_{\textrm{loc}}((0,\infty))^d$ satisfy that $\Phi(0)=\sigma(0)=\nu(0)=0$, that $\Phi'(\xi)>0$ for every $\xi\in(0,\infty)$, and the following three properties.
\begin{enumerate}[(i)]
\item There exists $c\in(0,\infty)$ such that
\begin{equation}\label{assume_f10}\limsup_{\xi\rightarrow 0^+}\frac{\sigma^2(\xi)}{\xi}\leq c\;\;\textrm{and}\;\;\sup_{\xi\in(0,\infty)}\frac{\Phi(\xi)}{\xi\Phi'(\xi)}\leq c.\end{equation}
\item There exists $c\in[1,\infty)$ such that
\begin{equation}\label{assume_f6} \Big(\sup\nolimits_{\xi'\in[0,\xi]}\sigma^2(\xi')\Big)\leq c(1+\xi+\sigma^2(\xi))\;\;\textrm{for every}\;\;\xi\in[0,\infty).\end{equation}
\item There exists $c\in[1,\infty)$ such that
\begin{equation}\label{assume_f7}\Big(\sup\nolimits_{\xi'\in[0,\xi]}\abs{\nu(\xi')}\Big)\leq c(1+\xi+\abs{\nu(\xi)})\;\;\textrm{for every}\;\;\xi\in[0,\infty).\end{equation}
\end{enumerate}
\end{assumption}

\begin{remark}  The role of assumption \eqref{assume_f6} is to guarantee that the approximations $\Theta_{\beta,M}(\rho)$ defined prior to \eqref{3_988} below converge $L^2_{\textrm{loc}}$-strongly to $\sigma(\rho)$.  Technically, it is used to guarantee the following condition:  for every $\rho\in L^\infty([0,T];L^1(\R^d))$ that satisfies $\sigma(\rho)\in L^2_{\textrm{loc}}(\R^d\times[0,T])$, we have that
\begin{equation}\label{assume_f1}\lim_{M\rightarrow\infty}\Big(\sup\nolimits_{\xi\in[M,(M+1)\wedge \rho]}\abs{\sigma(\xi)}\mathbf{1}_{\{\rho>M\}}\Big)=0\end{equation}
strongly in $L^2_{\textrm{loc}}(\R^d\times[0,T])$, which follows from \eqref{assume_f6} and an application of Chebyshev's inequality.  Assumption \eqref{assume_f6} could be replaced by the somewhat more general condition \eqref{assume_f1} with no change to the arguments.  Assumption~\eqref{assume_f7} is used in the identical way for $\nu$.  Technically, it guarantees that for every $\rho\in L^\infty([0,T];L^1(\R^d))$ that satisfies $\nu(\rho)\in L^1_\textrm{loc}([0,T];L^1(\R^d))$ we have that
\[\lim_{M\rightarrow\infty}\Big(\sup\nolimits_{\xi\in[M,(M+1)\wedge \rho]}\abs{\nu(\xi)}\mathbf{1}_{\{\rho>M\}}\Big)=0\;\;\textrm{strongly in}\;\;L^1_{\textrm{loc}}(\R^d\times[0,T]).\]
In this case the $L^1_{\textrm{loc}}$-integrability suffices, where for $\sigma$ the $L^2_{\textrm{loc}}$-integrability is used to treat certain stochastic integrals.  Similarly to the above, Assumption \eqref{assume_f7} could be replaced by this somewhat more general condition. \end{remark}

\begin{prop}\label{prop_measure_unique}  If $\rho$ is a stochastic kinetic solution of \eqref{i_1} in the sense of Definition~\ref{d_sol} with nonnegative, $\F_0$-measurable initial data $\rho_0\in L^1(\O;L^1(\R^d))$, it follows $\P$-a.s.\ that
\[\lim_{\beta\rightarrow 0}\Big(\beta^{-1}q(\R^d\times [\nicefrac{\beta}{2},\beta]\times[0,T])\Big)= 0.\]
\end{prop}

\begin{proof}  Since we have for every $\beta\in(0,1)$ that
\[\beta^{-1}q(\R^d\times [\nicefrac{\beta}{2},\beta]\times[0,T])\leq \int_0^T\int_{\frac{\beta}{2}}^\beta\int_{\R^d}\xi^{-1}q<\int_0^T\int_0^\infty\int_{\R^d}\frac{1}{\xi}q,\]
the claim follows from the dominated convergence theorem.\end{proof}

\begin{thm}\label{thm_unique}  Let $\xi^{\mathbf{a}}$ satisfy Assumption~\ref{a_random}, let $\Phi$, $\sigma$, and $\nu$ satisfy Assumption~\ref{assume_1}, let $\gamma\in(0,\infty)$, and let $\rho^1_0,\rho^2_0\in L^1(\O;(L^1_\gamma\cap\Ent_{\Phi,\gamma})(\R^d))$ be $\F_0$-measurable.  If $\rho^1,\rho^2$ are stochastic kinetic solutions of \eqref{i_1} in the sense of Definition~\ref{d_sol} with initial data $\rho_0^1,\rho_0^2$ respectively, then $\P$-a.s.\
\[\sup_{t\in[0,T]}\norm{\rho^1(\cdot,t)-\rho^2(\cdot,t)}_{L^1(\R^d)}\leq\norm{\rho^1_0-\rho^2_0}_{L^1(\R^d)}.\]
Furthermore, if $\rho_0\in L^1(\O;\Ent_{\Phi,\gamma}(\R^d))$ is $\F_0$-measurable then any two stochastic kinetic solutions $\rho^1$ and $\rho^2$ of \eqref{i_1} in the sense of Definition~\ref{d_sol} satisfy $\P$-a.s.\ that $\rho^1=\rho^2$ in $\Ent_{\Phi,\gamma}(\R^d)$ for every $t\in[0,T]$.
\end{thm}

\begin{proof}   Let $\chi^1$ and $\chi^2$ be the kinetic functions of $\rho^1$ and $\rho^2$ and for every $\ve,\d\in(0,1)$ and $i\in\{1,2\}$ let $\chi^{\ve,\d}_{t,i}(y,\eta)=(\chi^i(\cdot,\cdot,t)*\kappa^{\ve,\d})(y,\eta)$ be defined by
\[\kappa^{\ve,\d}(x,y,\xi,\eta) = \kappa^\ve(x-y)\kappa^\d(\xi-\eta),\]
for $\kappa^\ve$ and $\kappa^\d$ standard convolution kernels of scale $\ve$ and $\d$ on $\R^d$ and $\R$ respectively.  We will also define cutoff functions in the spatial and velocity variables.  For every $\b\in(0,1)$ let $\varphi_\beta$ satisfy $\varphi_\beta(\xi)=0$ if $\xi\in[0,\nicefrac{\beta}{2}]$, $\varphi_\beta(\xi)=1$ if $\xi\in[\beta,\infty)$, and $\varphi_\beta$ linearly interpolates between $0$ and $1$ on $[\nicefrac{\beta}{2},\beta]$, for every $M\in[2,\infty)$ let $\zeta_M(\xi)=1$ if $\xi\in[0,M]$, $\zeta_M(\xi)=0$ if $\xi\in[M+1,\infty)$, and let $\zeta_M$ linearly interpolate between $0$ and $1$ on $[M,M+1]$, and for every $R\in[1,\infty)$ let $\a_R$ be a smooth cutoff function of $B_R$ in $B_{2R}$ with $R\abs{\nabla^2\a_R}+\abs{\nabla\a_R}\leq\nicefrac{c}{R}$ for $c\in(0,\infty)$ independent of $R\in[1,\infty)$.

The proof of uniqueness is based on differentiating the following equality.  Due to the fact that the kinetic functions $\chi^i$ are $\{0,1\}$-valued, we have for every $t\in[0,T]$ that
\begin{align}\label{3_000}
\norm{\rho^1(x,t)-\rho^2(x,t)}_{L^1(\R^d)} & = \norm{\chi^1_t(\xi,x)-\chi^2_t(\xi,x)}_{L^2(\R^d\times\R)}
\\ \nonumber &  = \int_\R\int_{\R^d}\chi^1_t(\xi,x)+\chi^2_t(\xi,x)-2\chi^1_t(\xi,x)\chi^2_t(\xi,x)\dx\dxi
\\ \nonumber & =\lim_{\beta,\ve,\d\rightarrow 0}\lim_{M,R\rightarrow\infty}\int_\R\int_{\R^d}\big(\chi^{\ve,\d}_{t,1}+\chi^{\ve,\d}_{t,2}-2\chi^{\ve,\d}_{t,1}\chi^{\ve,\d}_{t,2}\big)\varphi_\beta\zeta_M\a_R,
\end{align}
where it follows from Definition~\ref{d_sol} and the Kolmogorov continuity criterion (see, for example, Revuz and Yor \cite[Chapter~1, Theorem~2.1]{RevYor1999}) that for every $\ve,\d\in(0,1)$ there exists a subset of full probability such that, for every $i\in\{1,2\}$, $(y,\eta)\in\R^d\times(\nicefrac{\d}{2},\infty)$, and $t\in[0,T]$, the regularized kinetic functions $\chi^{\ve,\d}_{t,i}$ satisfy, for $\overline{\kappa}^{\ve,\d}_{s,i}(x,y,\eta)=\kappa^{\ve,\d}(x,y,\rho^i(x,s),\eta)$,
\begin{align}\label{3_6}
& \chi^{\ve,\d}_{s,i}(y,\eta)\rvert_{s=0}^t=\nabla_y\cdot \Big(\int_0^t(\Phi'(\rho^i)\nabla\rho^i *\overline{\kappa}^{\ve,\d}_{s,i})(y,\eta)\Big)+  \partial_\eta\Big(\int_0^t (\kappa^{\ve,\d}*\dd q^i)(y,\eta)\Big)
\\ \nonumber & \quad + \nabla_y\cdot\Big(\frac{\langle\xi^\mathbf{a}\rangle_1}{2}\int_0^t \sigma'(\rho^i)^2\nabla\rho^i*\overline{\kappa}^{\ve,\d}_{s,i}(y,\eta)\Big) - \partial_\eta\Big(\int_0^t\langle \nabla\cdot\xi^\mathbf{a}\rangle_1 \sigma(\rho^i)^2 *\overline{\kappa}^{\ve,\d}_{s,i}(y,\eta)\Big)
\\ \nonumber & \quad - \int_0^t(\overline{\kappa}^{\ve,\d}_{s,i}*\nabla\cdot\nu(\rho^i))(y,\eta) - \int_0^t(\kappa^{\ve,\d}_{s,i}*\nabla\cdot(\sigma(\rho^i)\cdot\dd\xi^{\mathbf{a}}))(y,\eta).
\end{align}\normalsize
We will first treat the analogues of the first two terms on the righthand side of \eqref{3_000}.  It follows $\P$-a.s.\ from \eqref{3_6} that, for every $\ve,\beta\in(0,1)$, $M\in\N$, $R\in(1,\infty)$, and $\d\in(0,\nicefrac{\beta}{4})$, for every $t\in[0,T]$ and $i\in\{1,2\}$,
\begin{equation}\label{3_8}
\int_\mathbb{R}\int_{\R^d}\chi^{\ve,\d}_{s,i}(y,\eta)\varphi_\beta(\eta)\zeta_M(\eta)\a_R(y)\dy\deta\rvert^t_{s=0} = I^{i,\textrm{cut}}_t+I^{i,\textrm{mart}}_t+I^{i,\textrm{cons}}_t
\end{equation}
for the cutoff term defined by
\begin{align*}
& I^{i,\textrm{cut}}_t =-\int_0^t\int_{\mathbb{R}}\int_{\R^d}(\kappa^{\ve,\d}*q^i)(y,\eta)\partial_\eta(\varphi_\beta(\eta)\zeta_M(\eta))\a_R(y)
 \\  & +\frac{1}{2}\int_0^t\int_\mathbb{R}\int_{\R^d}(\langle \nabla\cdot\xi^\mathbf{a}\rangle_1 \sigma(\rho^i)^2*\overline{\kappa}^{\ve,\d}_{s,i})(y,\eta)\partial_\eta(\varphi_\beta(\eta)\zeta_M(\eta))\a_R(y)
 \\ & -\int_0^T\int_\R\int_{\R^d}(\Phi'(\rho^i)\nabla\rho^i *\overline{\kappa}^{\ve,\d}_{s,i})(y,\eta)\cdot\nabla\a_R(y)\varphi_\beta(\eta)\zeta_M(\eta)
 \\ & -\frac{\langle\xi^\mathbf{a}\rangle_1}{2}\int_0^t\int_\R\int_{\R^d}(\sigma'(\rho^i)^2\nabla\rho^i*\overline{\kappa}^{\ve,\d}_{s,i})(y,\eta)\cdot\nabla\a_R(y)\varphi_\beta(\eta)\zeta_M(\eta),
 \end{align*}
and for the martingale and conservative terms defined by
\begin{align*} & I^{i,\textrm{mart}}_t = -\int_0^t\int_\mathbb{R}\int_{\R^d}(\overline{\kappa}^{\ve,\d}_{s,i}*\nabla\cdot(\sigma(\rho^i(s))\dd\xi^{\mathbf{a}}))(y,\eta)\varphi_\beta(\eta)\zeta_M(\eta)\a_R(y),
\\ & I^{i,\textrm{cons}}_t=-\int_0^t\int_\R\int_{\R^d}(\overline{\kappa}^{\ve,\d}_{s,i}*\nabla\cdot\nu(\rho))(y,\eta)\varphi_\beta(\eta)\zeta_M(\eta)\a_R(y),
\end{align*}
where we emphasize that the terms $I^{i,\textrm{cut}}_t$, $I^{i,\textrm{mart}}_t$, and $I^{i,\textrm{cons}}_t$ depend on $\ve,\d,\beta\in(0,1)$, $M\in\N$, and $R\in(1,\infty)$.

We will now treat the mixed term of \eqref{3_000}.  From \eqref{3_6} and the stochastic product rule we have $\P$-a.s.\ that, for every $t\in[0,T]$ and $(y,\eta)\in\R^d\times\mathbb{R}$,
\begin{align}\label{3_09}
\chi^{\ve,\d}_{s,1}(y,\eta)\chi^{\ve,\d}_{s,2}(y,\eta)\rvert_{s=0}^t & = \chi^{\ve,\d}_{s,2}(y,\eta)\dd \chi^{\ve,\d}_{s,1}(y,\eta)+\chi^{\ve,\d}_{s,1}(y,\eta)\dd \chi^{\ve,\d}_{s,2}(y,\eta)
\\ \nonumber & \quad + \dd\langle \chi^{\ve,\d}_2,\chi^{\ve,\d}_1\rangle_s.
\end{align}
It follows from \eqref{3_6} and the definition of $\varphi_\beta$ that, for every $\ve,\beta\in(0,1)$, $M\in\N$, $R\in(1,\infty)$, and $\d\in(0,\nicefrac{\beta}{4})$,
\begin{align}\label{3_0009}
& \int_0^t\int_\mathbb{R}\int_{\R^d}\chi^{\ve,\d}_{s,2}(y,\eta)\dd \chi^{\ve,\d}_{s,1}(y,\eta)\varphi_\beta(\eta)\zeta_M(\eta)\dy\deta\ds
\\ \nonumber & = I^{2,1,\textrm{err}}_t+I^{2,1,\textrm{meas}}_t+I^{2,1,\textrm{cut}}_t+I^{2,1,\textrm{mart}}_t+I^{2,1,\textrm{cons}}_t, \end{align}
where, after adding the second term of \eqref{3_000009} below and subtracting it in \eqref{3_0000009} below, for $\overline{\kappa}^\d_{i,s}(y,\eta)=\kappa^\d_1(\rho^i(y,s)-\eta)$, the error term is defined by
\begin{align}\label{3_000009}
& I^{2,1,\textrm{err}}_t  = -\int_0^t\int_{\mathbb{R}}\int_{\R^d}(\Phi'(\rho^1)\nabla\rho^1*\overline{\kappa}^{\ve,\d}_{s,1})\cdot(\nabla\rho^2*\overline{\kappa}^{\ve,\d}_{s,2})\varphi_\beta(\eta)\zeta_M(\eta)\a_R(y)
\\ \nonumber & \quad +\int_0^t\int_{\mathbb{R}}\int_{\R^d}((\Phi'(\rho^1))^\frac{1}{2}\nabla\rho^1*\overline{\kappa}^{\ve,\d}_{s,1})\cdot((\Phi'(\rho^2))^\frac{1}{2}\nabla\rho^2*\overline{\kappa}^{\ve,\d}_{s,2})\varphi_\beta(\eta)\zeta_M(\eta)\a_R(y)
\\ \nonumber & \quad - \frac{\langle\xi^\mathbf{a}\rangle_1}{2}\int_0^t\int_\mathbb{R}\int_{\R^d}(\sigma'(\rho^1)^2\nabla\rho^1*\overline{\kappa}^{\ve,\d}_{s,1})\cdot(\nabla\rho^2*\overline{\kappa}^{\ve,\d}_{s,2})\varphi_\beta(\eta)\zeta_M(\eta)\a_R(y)
\\ \nonumber & \quad - \frac{1}{2}\int_0^t\int_\mathbb{R}\int_{\R^d}(\langle \nabla\cdot\xi^\mathbf{a}\rangle_1  \sigma(\rho^1)^2*\overline{\kappa}^{\ve,\d}_{s,1})\overline{\kappa}^{\d}_{s,2}\varphi_\beta(\eta)\zeta_M(\eta)\a_R(y),
\end{align}
the measure term is
\begin{align}\label{3_0000009}
 & I^{2,1,\textrm{meas}}_t  = \int_0^t\int_{\mathbb{R}}\int_{\R^d}(\kappa^{\ve,\d}*q^1)\overline{\kappa}^{\d}_{s,2}\varphi_\beta(\eta)\zeta_M(\eta)\a_R(y)
\\ \nonumber & \quad - \int_0^t\int_{\mathbb{R}}\int_{\R^d}((\Phi'(\rho^1))^\frac{1}{2}\nabla\rho^1*\overline{\kappa}^{\ve,\d}_{s,1})\cdot((\Phi'(\rho^2))^\frac{1}{2}\nabla\rho^2*\overline{\kappa}^{\ve,\d}_{s,2})\varphi_\beta(\eta)\zeta_M(\eta)\a_R(y),
\end{align}
the cutoff term is
\begin{align*}
& I^{2,1,\textrm{cut}}_t =-\int_0^t\int_{\mathbb{R}}\int_{\R^d}(\kappa^{\ve,\d}*q^1)\chi^{\ve,\d}_{s,2}\partial_\eta(\varphi_\beta\zeta_M)\a_R
\\ & \quad +\frac{1}{2}\int_0^t\int_\mathbb{R}\int_{\R^d}(\langle \nabla\cdot\xi^\mathbf{a}\rangle_1 \sigma(\rho^1)^2*\kappa^{\ve,\d}_{s,1})\chi^{\ve,\d}_{s,2}\partial_\eta(\varphi_\beta\zeta_M)\a_R
 \\ & \quad -\int_0^T\int_\R\int_{\R^d}(\Phi'(\rho^1)\nabla\rho^1*\overline{\kappa}^{\ve,\d}_{s,1})\chi^{\ve,\d}_{s,2}\cdot\nabla\a_R\varphi_\beta\zeta_M
 \\ & \quad -\frac{\langle\xi^\mathbf{a}\rangle_1}{2}\int_0^t\int_\R\int_{\R^d}(\sigma'(\rho^1)^2\nabla\rho^1*\overline{\kappa}^{\ve,\d}_{s,1})\chi^{\ve,\d}_{s,2}\cdot\nabla\a_R\varphi_\beta\zeta_M,
 \end{align*}
the martingale term is $I^{2,1,\textrm{mart}}_t=-\int_0^t\int_\mathbb{R}\int_{\R^d}(\overline{\kappa}^{\ve,\d}_{s,1}*\nabla\cdot(\sigma(\rho^1)\dd\xi^{\mathbf{a}}))\chi^{\ve,\d}_{s,2}\varphi_\beta\zeta_M\a_R$, and the conservative term is $I^{2,1,\textrm{cons}}_t = -\int_0^t\int_\mathbb{R}\int_{\R^d}(\overline{\kappa}^{\ve,\d}_{s,1}*\nabla\cdot\nu(\rho^1))\chi^{\ve,\d}_{s,2}\varphi_\beta\zeta_M\a_R$.
The analogous formula holds for the second term on the righthand side of \eqref{3_09}.  For the final term of \eqref{3_09}, it follows from \eqref{3_6} and the definition of $\xi^{\mathbf{a}}$ that, summing over $k\in\N$,
\begin{align}\label{3_009}
 \int_0^t\int_\mathbb{R}\int_{\R^d}\dd\langle\chi^{\ve,\d}_1,\chi^{\ve,\d}_{s,2}\rangle_s(y,\eta)\varphi_\beta\zeta_M\a_R & = \int_0^t \int_\mathbb{R}\int_{\R^d}( f_k\nabla\sigma(\rho^1)*\overline{\kappa}^{\ve,\d}_{s,1})\cdot(f_k\nabla\sigma(\rho^2)*\overline{\kappa}^{\ve,\d}_{s,2})\varphi_\beta\zeta_M\a_R
\\ \nonumber  & \quad +  \int_0^t \int_\mathbb{R}\int_{\R^d}(\sigma(\rho^1)\nabla f_k*\overline{\kappa}^{\ve,\d}_{s,1})\cdot(\sigma(\rho^2)\nabla f_k*\overline{\kappa}^{\ve,\d}_{s,2})\varphi_\beta\zeta_M\a_R
\\ \nonumber  & \quad + \int_0^t \int_\mathbb{R}\int_{\R^d}(\nabla\sigma(\rho^1)f_k*\overline{\kappa}^{\ve,\d}_{s,1})\cdot(\sigma(\rho^2)\nabla f_k*\overline{\kappa}^{\ve,\d}_{s,2})\varphi_\beta\zeta_M\a_R
\\ \nonumber  & \quad +   \int_0^t \int_\mathbb{R}\int_{\R^d}(\sigma(\rho^1)\nabla f_k*\overline{\kappa}^{\ve,\d}_{s,1}))\cdot(f_k\nabla\sigma(\rho^2)*\overline{\kappa}^{\ve,\d}_{s,2})\varphi_\beta\zeta_M\a_R.
\end{align}
It follows from \eqref{3_09}, \eqref{3_0009}, and \eqref{3_009} that
\begin{align*}
& \int_\mathbb{R}\int_{\R^d}\chi^{\ve,\d}_{s,1}(y,\eta)\chi^{\ve,\d}_{s,2}(y,\eta)\varphi_\beta\zeta_M\a_R\dy\deta\rvert^t_{s=0}  = I^{\textrm{err}}_t+I^{\textrm{meas}}_t+I^{\textrm{mix},\textrm{cut}}_t+I^{\textrm{mix},\textrm{mart}}_t+I^{\textrm{mix},\textrm{cons}}_t,
\end{align*}
where the error terms \eqref{3_000009} and \eqref{3_009} combine to form,
\begin{align*}
& I^{\textrm{err}}_t = -\int_0^t\int_\mathbb{R}\int_{(\R^d)^3}\big((\Phi'(\rho^1))^\frac{1}{2}-(\Phi'(\rho^2))^\frac{1}{2}\big)^2\nabla\rho^1\cdot\nabla\rho^2\overline{\kappa}^{\ve,\d}_{s,1}\overline{\kappa}^{\ve,\d}_{s,2}\varphi_\beta\zeta_M\a_R
\\ \nonumber & - \frac{\langle\xi^\mathbf{a}\rangle_1}{2}\int_0^t\int_\mathbb{R}\int_{\R^d}\big((\sigma'(\rho^1)^2\nabla\rho^1*\overline{\kappa}^{\ve,\d}_{s,1})\cdot(\nabla\rho^2*\overline{\kappa}^{\ve,\d}_{s,2})+(\sigma'(\rho^2)^2\nabla\rho^2*\overline{\kappa}^{\ve,\d}_{s,2})\cdot(\nabla\rho^1*\overline{\kappa}^{\ve,\d}_{s,1})\big)\varphi_\beta\zeta_M\a_R
\\ \nonumber & - \frac{1}{2}\int_0^t\int_\mathbb{R}\int_{\R^d}(\langle \nabla\cdot\xi^\mathbf{a}\rangle \sigma(\rho^1)^2*\overline{\kappa}^{\ve,\d}_{s,1})\overline{\kappa}^{\d}_{s,2}\varphi_\beta\zeta_M\a_R - \frac{1}{2}\int_0^t\int_\mathbb{R}\int_{\R^d}(\langle \nabla\cdot\xi^\mathbf{a}\rangle \sigma(\rho^2)^2*\overline{\kappa}^{\ve,\d}_{s,2})\overline{\kappa}^{\d}_{s,1}\varphi_\beta\zeta_M\a_R
\\ \nonumber & + \int_0^t \int_\mathbb{R}\int_{\R^d}(f_k\nabla\sigma(\rho^1)*\overline{\kappa}^{\ve,\d}_{s,1})\cdot (f_k\nabla\sigma(\rho^2)*\overline{\kappa}^{\ve,\d}_{s,2})\varphi_\beta\zeta_M\a_R
\\ \nonumber  & +  \int_0^t \int_\mathbb{R}\int_{\R^d}(\sigma(\rho^1)\nabla f_k*\overline{\kappa}^{\ve,\d}_{s,1})\cdot(\sigma(\rho^2)\nabla f_k*\overline{\kappa}^{\ve,\d}_{s,2})\varphi_\beta\zeta_M\a_R
\\ \nonumber  & + \int_0^t \int_\mathbb{R}\int_{\R^d}(\nabla\sigma(\rho^1)f_k*\overline{\kappa}^{\ve,\d}_{s,1})\cdot(\sigma(\rho^2)\nabla f_k*\overline{\kappa}^{\ve,\d}_{s,2})\varphi_\beta\zeta_M\a_R
\\ \nonumber  & +   \int_0^t \int_\mathbb{R}\int_{\R^d}(\sigma(\rho^1)\nabla f_k*\overline{\kappa}^{\ve,\d}_{s,1}))\cdot (f_k\nabla\sigma(\rho^2)*\overline{\kappa}^{\ve,\d}_{s,2})\varphi_\beta\zeta_M\a_R,
\end{align*}
and where the measure terms \eqref{3_0000009} combine to form
\begin{align*}
I^{\textrm{meas}}_t & =  \int_0^t\int_{\mathbb{R}}\int_{\R^d}(\kappa^{\ve,\d}*q^1)\overline{\kappa}^{\d}_{s,2}\varphi_\beta\zeta_M\a_R +\int_0^t\int_{\mathbb{R}}\int_{\R^d}(\kappa^{\ve,\d}*q^2)\overline{\kappa}^{\d}_{s,1}\varphi_\beta\zeta_M\a_R
\\ \nonumber & \quad -2\int_0^t\int_{\mathbb{R}}\int_{(\R^d)^3}(\Phi'(\rho^1))^\frac{1}{2}(\Phi'(\rho^2))^\frac{1}{2}\nabla\rho^1\cdot\nabla\rho^2\overline{\kappa}^{\ve,\d}_{s,1}\overline{\kappa}^{\ve,\d}_{s,2}\varphi_\beta\zeta_M\a_R.
\end{align*}
For the cutoff, martingale, and conservative terms defined respectively by $I^{\textrm{cut}}_t= I^{1,\textrm{cut}}_t+I^{2,\textrm{cut}}_t-2(I^{2,1,\textrm{cut}}_t+I^{1,2,\textrm{cut}}_t)$, and similarly for $I^{\textrm{mart}}_t$ and $I^{\textrm{cons}}_t$, we have from \eqref{3_8} and \eqref{3_0009} that, $\P$-a.s.\ for every $t\in[0,T]$,
\begin{align}\label{3_21}
& \int_\mathbb{R}\int_{\R^d}\Big(\chi^{\ve,\d}_{s,1}+\chi^{\ve,\d}_{s,2}-2\chi^{\ve,\d}_{s,1}\chi^{\ve,\d}_{s,2}\Big)\varphi_\beta\zeta_M\rvert_{s=0}^t
\\ \nonumber & =-2I^{\textrm{err}}_t-2I^{\textrm{meas}}_t+I^{\textrm{mart}}_t+I^{\textrm{cut}}_t+I^{\textrm{cons}}_t.
\end{align}
We will handle the five terms on the righthand side of \eqref{3_21} separately.

\textbf{The measure term}.  It follows from the regularity property \eqref{2_5000} of the kinetic measures and H\"older's inequality that the measure term $\P$-a.s.\ satisfies, for every $t\in[0,T]$,
\begin{equation}\label{3_22} I^{\textrm{meas}}_t\geq 0.\end{equation}

\textbf{The error term}.  For the error term, a repetition of the analysis leading from \cite[Equation~(4.16)]{FG21} to \cite[Equation~(4.18)]{FG21} using the probabilistic stationarity of the noise proves that, after passing to the limit $\ve\rightarrow 0$, $\P$-a.s.\ for every $t\in[0,T]$,
\begin{align*}
 \lim_{\ve\rightarrow 0}I^{\textrm{err}}_t & = -\int_0^t\int_\mathbb{R}\int_{\R^d}\Big((\Phi'(\rho^1))^\frac{1}{2}-(\Phi'(\rho^2))^\frac{1}{2}\Big)^2\nabla\rho^1\cdot\nabla\rho^2\overline{\kappa}^{\d}_{s,1}\overline{\kappa}^{\d}_{s,2}\varphi_\beta\zeta_M\a_R
\\ \nonumber & \quad - \frac{1}{2}\int_0^t\int_\mathbb{R}\int_{\R^d}\langle \nabla\cdot\xi^\mathbf{a}\rangle_1 (\sigma(\rho^1)-\sigma(\rho^2))^2\overline{\kappa}^{\d}_{s,1}\overline{\kappa}^{\d}_{s,2}\varphi_\beta\zeta_M\a_R
\\ \nonumber & \quad - \frac{\langle\xi^\mathbf{a}\rangle_1}{2}\int_0^t\int_\mathbb{R}\int_{\R^d}(\sigma'(\rho^1)-\sigma'(\rho^2))^2\nabla\rho^1\cdot\nabla\rho^2\overline{\kappa}^{\d}_{s,1}\overline{\kappa}^{\d}_{s,2}\varphi_\beta\zeta_M\a_R.
\end{align*}
We therefore have using the local Lipschitz continuity of $\Phi'$, $\sigma$, and $\sigma'$ in Assumption~\ref{assume_1} that, for some $c\in(0,\infty)$ depending $\beta$, $M$, $\langle\xi^\mathbf{a}\rangle_1$, and $\langle \nabla\cdot\xi^\mathbf{a}\rangle_1 $,
\begin{align*}
 \lim_{\ve\rightarrow 0}I^{\textrm{err}}_t & \leq c\d\int_0^t\int_\R\int_{\R^d}\mathbf{1}_{\{0<\abs{\rho^1-\rho^2}<c\d\}}\abs{\nabla\rho^1}\abs{\nabla\rho^2}(\d\overline{\kappa}^{\d}_{s,1})\overline{\kappa}^{\d}_{s,2}\varphi_\beta\zeta_M\a_R
 \\ & \quad +c\d\int_0^t\int_\R\int_{\R^d}\mathbf{1}_{\{0<\abs{\rho^1-\rho^2}<c\d\}}(\d\overline{\kappa}^{\d}_{s,1})\overline{\kappa}^{\d}_{s,2}\varphi_\beta\zeta_M\a_R.
 \end{align*}
 It then follows from the boundedness of $\d\overline{\kappa}^\d_{s,i}$, the compact support of $\a_R$, the local regularity of the $\rho^i$ and the regularity of $\sigma$ in Assumption~\ref{assume_1}, and the dominated convergence theorem that
\begin{equation}\label{3_12}\limsup_{\d\rightarrow 0}\big(\limsup_{\ve\rightarrow 0}\abs{I^{\textrm{err}}_t}\big)=0,\end{equation}
which completes the analysis of the error term.

\textbf{The martingale term}.  For the martingale term, a repetition of the analysis leading from \cite[Equation~(4.19)]{FG21} to \cite[Equation~(4.23)]{FG21} proves that, $\P$-a.s.\ along a deterministic subsequence $\ve,\d\rightarrow 0$,
\begin{align}\label{3_987}
\lim_{\d,\ve\rightarrow 0}\big(I^\textrm{mart}_t\big)  & =  \int_0^t\int_{\R^d}\sgn(\rho^2-\rho^1)\varphi_\beta(\rho^1)\zeta_M(\rho^1)\a_R(y)\nabla\cdot(\sigma(\rho^1)\dd\xi^{\mathbf{a}})
\\ \nonumber & \quad - \int_0^t\int_{\R^d}\sgn(\rho^2-\rho^1)\varphi_\beta(\rho^2)\zeta_M(\rho^2)\a_R(y)\nabla\cdot(\sigma(\rho^2)\dd\xi^{\mathbf{a}}).
\end{align}
For every $\beta\in(0,1)$ and $M\in\N$ let $\Theta_{\beta,M}\colon[0,\infty)\rightarrow\R$ be the unique function that satisfies $\Theta_{\beta,M}(0)=0$ and $\Theta_{\beta,M}'(\xi)=\varphi_\beta(\xi)\zeta_M(\xi)\sigma'(\xi)$.  Returning to \eqref{3_987}, it follows that, along subsequences,
\begin{align*}
& \lim_{\d,\ve\rightarrow 0}\big(I^\textrm{mart}_t\big)  =  \int_0^t\int_{\R^d}\sgn(\rho^2-\rho^1)\a_R(y)\nabla\cdot\big(\big(\Theta_{\beta,M}(\rho^1,x)-\Theta_{\beta,M}(\rho^2,x)\big)\dd\xi^{\mathbf{a}}\big)
\\ &  \quad +\int_0^t\int_{\R^d}\sgn(\rho^2-\rho^1)\a_R(y)\Big(\varphi_\beta(\rho^1)\zeta_M(\rho^1)\sigma(\rho^1)-\Theta_{\beta,M}(\rho^1)\Big)\nabla\cdot\dd\xi^{\mathbf{a}}
\\ & \quad - \int_0^t\int_{\R^d}\sgn(\rho^2-\rho^1)\a_R(y)\Big(\varphi_\beta(\rho^2)\zeta_M(\rho^2)\sigma(\rho^2)-\Theta_{\beta,M}(\rho^2)\Big)\nabla\cdot\dd\xi^{\mathbf{a}}.
\end{align*}
It follows from the local regularity of the $\rho^i$, the Lipschitz continuity and boundedness of $\Theta_{\beta,M}$ and $\varphi_\beta\zeta_M\sigma$, the $L^1$-integrability of the $\rho^i$, and the Burkholder--Davis--Gundy inequality that, $\P$-a.s.\ for every $t\in[0,T]$, along a deterministic subsequence $R\rightarrow\infty$,
\begin{align}\label{3_988}
& \lim_{R\rightarrow\infty}\lim_{\d,\ve\rightarrow 0}\big(I^\textrm{mart}_t\big) =  \int_0^t\int_{\R^d}\sgn(\rho^2-\rho^1)\nabla\cdot\big(\big(\Theta_{\beta,M}(\rho^1,x)-\Theta_{\beta,M}(\rho^2,x)\big)\dd\xi^{\mathbf{a}}\big)
\\ \nonumber &  \quad +\int_0^t\int_{\R^d}\sgn(\rho^2-\rho^1)\big(\varphi_\beta(\rho^1)\zeta_M(\rho^1)\sigma(\rho^1,x)-\Theta_{\beta,M}(\rho^1,x)\big)\nabla\cdot\dd\xi^{\mathbf{a}}
\\ \nonumber & \quad - \int_0^t\int_{\R^d}\sgn(\rho^2-\rho^1)\big(\varphi_\beta(\rho^2)\zeta_M(\rho^2)\sigma(\rho^2,x)-\Theta_{\beta,M}(\rho^2,x)\big)\nabla\cdot\dd\xi^{\mathbf{a}}.
\end{align}
It follows from the global Lipschitz continuity of $\Theta_{\beta,M}$ in the $\xi$-variable that the first term on the righthand side of \eqref{3_988} satisfies
\[  \int_0^t\int_{\R^d}\sgn(\rho^2-\rho^1)\nabla\cdot\big(\big(\Theta_{\beta,M}(\rho^1)-\Theta_{\beta,M}(\rho^2)\big)\dd\xi^{\mathbf{a}}\big)=0.\]
For the second and third terms on the righthand side of \eqref{3_988}, it follows from the definitions and an integration by parts that, for $c\in(0,\infty)$ independent of $\beta$ and $M$,
\[\abs{\Theta_{\beta,M}(\rho^i)-\sigma(\rho^i)}\leq c\big(\sup_{\xi\in[0,\beta\wedge \rho^i]}\abs{\sigma(\xi)}+\abs{\sigma(\rho^i)}\mathbf{1}_{\{\rho^i>M\}}+\sup_{\xi\in[M,(M+1)\wedge\rho^i]}\abs{\sigma(\xi)}\mathbf{1}_{\{\rho^i>M\}}\big).\]
 It then follows from the $L^2$-integrability of the $\sigma(\rho^i)$, the continuity of $\sigma$ and $\sigma(0)=0$, and \eqref{assume_f1} that
\[\lim_{M\rightarrow\infty}\lim_{\beta\rightarrow 0}\varphi_\beta(\rho^i)\zeta_M(\rho^i)\sigma(\rho^i)=\lim_{M\rightarrow\infty}\lim_{\beta\rightarrow 0}\Theta_{\beta,M}(\rho^i)=\sigma(\rho^i)\;\;\textrm{strongly in}\;\;L^2_{\textrm{loc}}(\R^d\times[0,T]).\]
Therefore, it follows from the Burkholder--Davis--Gundy inequality that, $\P$-a.s.\ along a deterministic subsequence $M\rightarrow\infty$ and $\beta\rightarrow 0$, for every $R\in(0,\infty)$,
\[\lim_{\beta\rightarrow 0}\lim_{M\rightarrow\infty}\max_{t\in[0,T]}\abs{\int_0^t\int_{\R^d}\sgn(\rho^2-\rho^1)\a_R(y)\big(\varphi_\beta(\rho^1)\zeta_M(\rho^1)\sigma(\rho^1)-\Theta_{\beta,M}(\rho^1)\big)\a_R(y)\nabla\cdot\dd\xi^{\mathbf{a}}}=0,\]
and similarly for the remaining three terms.  We therefore have that
\begin{equation}\label{3_29} \lim_{\beta\rightarrow 0}\lim_{M\rightarrow\infty} \lim_{R\rightarrow\infty}\lim_{\d\rightarrow 0}\lim_{\ve\rightarrow 0} \big(\max_{t\in[0,T]}I^\textrm{mart}_t\big)=0,\end{equation}
which completes the analysis of the martingale term.

\textbf{The conservative term.}  A virtually identical analysis to that leading to \cite[Equation~(4.29)]{FG21} proves that the conservative term satisfies
\begin{equation}\label{3_29292929}  \lim_{R\rightarrow\infty}\lim_{\beta\rightarrow 0}\lim_{M\rightarrow\infty}\lim_{\d\rightarrow 0}\lim_{\ve\rightarrow 0} \big(\max_{t\in[0,T]}I^\textrm{cons}_t\big)=0,\end{equation}
where in this case we require only the $L^1_{\textrm{loc}}$-integrability of $\nu(\rho^i)$ to apply the dominated convergence theorem.  The $L^2_{\textrm{loc}}$-integrability of $\sigma(\rho^i)$ is used to apply the Burkholder--Davis--Gundy inequality to treat the stochastic integral.

\textbf{The cutoff term.}  For the cutoff term, the analysis leading to \eqref{3_987} proves that, for every $t\in[0,T]$, $\ve,\beta\in(0,1)$, $\d\in(0,\nicefrac{\beta}{4})$, $M\in\N$, and $R\in(1,\infty)$,
\begin{align}\label{cutoff_term}
& \lim_{\ve,\d\rightarrow 0} I^{\textrm{cut}}_t  \leq \int_0^t\int_{\mathbb{R}}\int_{\R^d}\partial_\eta(\varphi_\beta\zeta_M)\a_R\big(\dd q^1+\dd q^2\big)
 \\ \nonumber  & -\frac{1}{2}\int_0^t\int_\mathbb{R}\int_{\R^d}\sgn(\rho^2-\rho^1)\langle \nabla\cdot\xi^\mathbf{a}\rangle_1 \big(\sigma(\rho^1)^2\partial_\eta(\varphi_\beta\zeta_M)(\rho^1)-\sigma(\rho^2)^2\partial_\eta(\varphi_\beta\zeta_M)(\rho^2)\big)\a_R
 \\ \nonumber & +\int_0^t\int_\R\int_{\R^d}\sgn(\rho^2-\rho^1)\big(\varphi_\beta\zeta_M(\rho^1)\nabla\Phi(\rho^1)-\varphi_\beta\zeta_M(\rho^2)\nabla\Phi(\rho^2)\big)\cdot\nabla\a_R
 \\ \nonumber & +\frac{\langle\xi^\mathbf{a}\rangle_1}{2}\int_0^t\int_\R\int_{\R^d}\sgn(\rho^2-\rho^1)\big(\sigma'(\rho^1)^2\varphi_\beta\zeta_M(\rho^1)\nabla\rho^1-\sigma'(\rho^2)^2\varphi_\beta\zeta_M(\rho^2)\nabla\rho^2\big)\cdot\nabla\a_R.
\end{align}
Therefore, for $\Phi_{\beta,M}'(\xi)=\Phi'(\xi)\varphi_\beta(\xi)\zeta_M(\xi)$ with $\Phi_{\beta,M}(0)=0$ and for $\Sigma_{\beta,M}'(\xi)=\sigma'(\xi)^2\varphi_\beta(\xi)\zeta_M(\xi)$ with $\Sigma_{\beta,M}(0)=0$ we have that, for some $c\in(0,\infty)$,
\begin{align*}
 \limsup_{\ve,\d\rightarrow 0} I^{\textrm{cut}}_t  & \leq c\int_0^t\int_{\mathbb{R}}\int_{\R^d}\Big(\frac{1}{\xi}\mathbf{1}_{\{\nicefrac{\beta}{2}<\xi<\beta\}}+\mathbf{1}_{\{M<\xi<M+1\}}\Big)\a_R\big(\dd q^1+\dd q^2\big)
  \\  & \quad  -\frac{1}{2}\int_0^t\int_\mathbb{R}\int_{\R^d}\sgn(\rho^2-\rho^1)\langle \nabla\cdot\xi^\mathbf{a}\rangle_1 \big(\sigma(\rho^1)^2\partial_\eta(\varphi_\beta\zeta_M)(\rho^1)-\sigma(\rho^2)^2\partial_\eta(\varphi_\beta\zeta_M)(\rho^2)\big)\a_R
  \\ & \quad -\int_0^t\int_{\R^d}\abs{\Phi_{\beta,M}(\rho^1)-\Phi_{\beta,M}(\rho^2)}\Delta\alpha_R -\frac{\langle\xi^\mathbf{a}\rangle_1}{2}\int_0^t\int_{\R^d}\abs{\Sigma_{\beta,M}(\rho^1)-\Sigma_{\beta,M}(\rho^2)}\Delta\alpha_R.
\end{align*}
We will first pass to the limit $R\rightarrow\infty$, for which it is necessary to consider separately the cases $d=1$, $d=2$, and $d\geq 3$.  If $d=1$, it follows immediately from the boundedness of $\Phi_{\beta,M}$ and $\Sigma_{\beta,M}$---which follows from their definitions---and $\abs{\Delta\alpha_R}\leq R^{-2}\mathbf{1}_{B_{2R}\setminus B_R}$ that
\begin{equation}\label{3_20}\lim_{R\rightarrow\infty}\Big(\int_0^t\int_{\R^d}\abs{\Phi_{\beta,M}(\rho^1)-\Phi_{\beta,M}(\rho^2)}\abs{\Delta\alpha_R}+\frac{\langle\xi^\mathbf{a}\rangle_1}{2}\int_0^t\int_{\R^d}\abs{\Sigma_{\beta,M}(\rho^1)-\Sigma_{\beta,M}(\rho^2)}\abs{\Delta\alpha_R}\Big)=0.\end{equation}
If $d=2$, due to the nondegeneracy and local $\C^1$-regularity of $\Phi$ in Assumption~\ref{assume_1}, for every $\eta\in(0,1)$ there exists $c_\eta\in(0,\infty)$ depending on $\eta$ such that
\begin{equation}\label{assume_f30} \abs{\rho^i-\gamma}\mathbf{1}_{\{\abs{\rho_i-\gamma}\geq\eta\}}\leq c_\eta\Psi_{\Phi,\gamma}(\rho^i).\end{equation}
Therefore, it follows from the entropy estimate that for every $\eta\in(0,1)$ there exists $c_\eta\in(0,\infty)$ depending on $\eta$ such that
\[\sup_{t\in[0,T]}\norm{(\rho-\gamma)\mathbf{1}_{\{\abs{\rho_i-\gamma}\geq\eta\}}}_{L^1(\R^d)}\leq c_\eta.\]
In $d=2$, therefore, using the local $\C^1$-regularity of $\Phi$ and $\sigma$ in Assumption~\ref{assume_1}, the bound $\abs{\Delta\alpha_R}\leq R^{-2}\mathbf{1}_{B_{2R}\setminus B_R}$, and the dominated convergence theorem, for every $\eta\in(0,1)$,
\[\limsup_{R\rightarrow\infty}\Big(\int_0^t\int_{\R^d}\abs{\Phi_{\beta,M}(\rho^1)-\Phi_{\beta,M}(\rho^2)}\abs{\Delta\alpha_R}+\frac{\langle\xi^\mathbf{a}\rangle_1}{2}\int_0^t\int_{\R^d}\abs{\Sigma_{\beta,M}(\rho^1)-\Sigma_{\beta,M}(\rho^2)}\abs{\Delta\alpha_R}\Big)\leq \eta,\]
from which we reach the conclusion of \eqref{3_20} in $d=2$ by taking $\eta\rightarrow 0$.  It remains to treat the case $d\geq 3$.  We first observe using the local $\C^1$-regularity of $\Phi$ and $\sigma$ in Assumption~\ref{assume_1} and the local nondegeneracy of $\Phi$ in \eqref{assume_f10} that there exists $c_{\beta,M}\in(0,\infty)$ such that
\[\abs{\Phi_{\beta,M}(\rho^i)-\Phi_{\beta,M}(\gamma)}+\abs{\Sigma_{\beta,M}(\rho^i)-\Sigma_{\beta,M}(\gamma)}\leq c_{\beta,M}\big(\abs{\Phi^\frac{1}{2}(\rho^i)-\Phi^\frac{1}{2}(\gamma)}^\frac{1}{2}+\abs{\rho^i-\gamma}\mathbf{1}_{\{\abs{\rho^i-\gamma}\geq \frac{\gamma}{2}\}}\big).\]
Similarly to the case $d=2$, since the $\abs{\rho^i-\gamma}\mathbf{1}_{\{\abs{\rho^i-\gamma}\geq \frac{\gamma}{2}\}}$ are $L^\infty_tL^1_x$-integrable it follows from the dominated convergence theorem and H\"older's inequality that, for $\nicefrac{1}{2_*}=\nicefrac{1}{2}-\nicefrac{1}{d}$,
\begin{align*}
& \limsup_{R\rightarrow\infty}\Big(\int_0^t\int_{\R^d}\abs{\Phi_{\beta,M}(\rho^1)-\Phi_{\beta,M}(\rho^2)}\abs{\Delta\alpha_R}+\frac{\langle\xi^\mathbf{a}\rangle_1}{2}\int_0^t\int_{\R^d}\abs{\Sigma_{\beta,M}(\rho^1)-\Sigma_{\beta,M}(\rho^2)}\abs{\Delta\alpha_R}\Big)
\\ & \leq \sum_{i=1}^2 c_{\beta,M}\int_0^t\int_{\R^d}\abs{\Phi^\frac{1}{2}(\rho^i)-\Phi^\frac{1}{2}(\gamma)}^\frac{1}{2}\abs{\Delta\a_R}
\\ & \leq \sum_{i=1}^2 \tilde{c}_{\beta,M}\int_0^t\Big(\int_{B_{2R}\setminus B_R}\abs{\Phi^\frac{1}{2}(\rho^i)-\Phi^\frac{1}{2}(\gamma)}^{2_*}\Big)^\frac{2}{2_*}\Big(\int_{\R^d}\abs{\Delta\a_R}^\frac{d}{2}\Big)^\frac{2}{d}.
\end{align*}
Since it follows from the Sobolev embedding theorem, which can be applied owing to the density of functions of the form $\gamma+\psi$ for $\psi\in\C^\infty_c(\R^d)$ in $\Ent_{\Phi,\gamma}(\R^d)$, that
\[\int_0^T\Big(\int_{\R^d}\abs{\Phi^\frac{1}{2}(\rho^i)-\Phi^\frac{1}{2}(\gamma)}^{2_*}\Big)^\frac{2}{2_*}\leq c\int_0^T\int_{\R^d}\abs{\nabla\Phi^\frac{1}{2}(\rho^i)}^2,\]
the conclusion of \eqref{3_20} for $d\geq 3$ follows from the entropy estimate, $\abs{\Delta\alpha_R}\leq R^{-2}\mathbf{1}_{B_{2R}\setminus B_R}$, and the dominated convergence theorem.

We therefore have that, after passing to the limit $R\rightarrow\infty$, $\P$-a.s.\ for every $t\in[0,T]$,
\begin{align*}
& \limsup_{R\rightarrow\infty}\limsup_{\ve,\d\rightarrow 0} I^{\textrm{cut}}_t   \leq \int_0^t\int_{\mathbb{R}}\partial_\eta(\varphi_\beta(\eta)\zeta_M(\eta))\big(\dd q^1+\dd q^2\big)
 \\  & \quad +\frac{1}{2}\int_0^t\int_{\R^d}\langle \nabla\cdot\xi^\mathbf{a}\rangle_1 \sgn(\rho^2-\rho^1)\big(\sigma^2(\rho^1)\partial_\eta(\varphi_\beta\zeta_M)(\rho^1)-\sigma^2(\rho^2)\partial_\eta(\varphi_\beta\zeta_M)(\rho^2)\big),
 \end{align*}
 and, therefore, using the definitions of the cutoff functions, the boundedness of $\langle \nabla\cdot\xi^\mathbf{a}\rangle_1 $, and the assumption that $\sigma^2(\xi)\leq c\xi$ for $\xi\in(0,1]$, for some $c\in(0,\infty)$,
\begin{align*}
& \limsup_{R\rightarrow\infty}\limsup_{\ve,\d\rightarrow 0} I^{\textrm{cut}}_t  \leq c\sum_{i=1}^2\big(\beta^{-1}q^i(\R^d\times[\nicefrac{\beta}{2},\beta]\times[0,T])+q^i(\R^d\times[M,M+1]\times[0,T])\big)
\\ & \quad +c\sum_{i=1}^2\int_0^t\int_{\R^d}\langle \nabla\cdot\xi^\mathbf{a}\rangle_1 \mathbf{1}_{\{\nicefrac{\beta}{2}\leq \rho^i\leq \beta\}} +\langle \nabla\cdot\xi^\mathbf{a}\rangle_1 (\sigma(\rho^i)-\sigma(\gamma))^2\mathbf{1}_{\{M\leq \rho^i\leq M+1\}}
\\ & \quad +c\sum_{i=1}^2\int_0^t\int_{\R^d}\langle \nabla\cdot\xi^\mathbf{a}\rangle_1 \sigma(\gamma)^2\mathbf{1}_{\{M\leq \rho^i\leq M+1\}}.
\end{align*}
It is then a consequence of $\langle \nabla\cdot\xi^\mathbf{a}\rangle_1 \in (L^1\cap L^\infty)(\R^d)$, Proposition~\ref{prop_measure_unique}, the vanishing of the measures at infinity in Definition~\ref{d_sol}, the integrability of the $(\rho^i-\gamma)\mathbf{1}_{\{\abs{\rho^i-\gamma}\geq\nicefrac{\gamma}{2}\}}$, and the $L^2$-integrability of $(\sigma(\rho^i)-\sigma(\gamma))$ that the righthand side vanishes along subsequences $M\rightarrow\infty$ and $\beta\rightarrow 0$.  We therefore have that, along subsequences,
\begin{equation}\label{3_18}\lim_{\beta\rightarrow 0}\lim_{M\rightarrow\infty} \lim_{R\rightarrow\infty}\lim_{\d\rightarrow 0}\lim_{\ve\rightarrow 0} \big(\max_{t\in[0,T]}I^\textrm{cut}_t\big)=0,\end{equation}
which completes the analysis of the cutoff term.

\textbf{Conclusion.} Properties of the kinetic function and estimates \eqref{3_21}, \eqref{3_22}, \eqref{3_12}, \eqref{3_29}, \eqref{3_29292929}, and \eqref{3_18} prove that there $\P$-a.s.\ exist random subsequences $\ve,\d,\beta\rightarrow 0$ and $R,M\rightarrow\infty$ such that, for every $t\in[0,T]$,
\begin{align*}
& \int_\mathbb{R}\int_{\R^d}\abs{\chi^1_s-\chi^2_s}^2\rvert_{s=0}^{s=t}
\\ & \leq \lim_{\beta\rightarrow 0}\lim_{M\rightarrow\infty}\lim_{R\rightarrow\infty}\lim_{\d\rightarrow 0}\lim_{\ve\rightarrow 0}\Big(-2I^{\textrm{err}}_t-2I^{\textrm{meas}}_t+I^{\textrm{mart}}_t+I^{\textrm{cut}}_t+I^{\textrm{cons}}_t\Big)=0.
\end{align*}
Properties of the kinetic function then prove that, for every $t\in[0,T]$,
\begin{align*}
\int_{\R^d}\abs{\rho^1(\cdot,t)-\rho^2(\cdot,t)} & =\int_\mathbb{R}\int_{\R^d}\abs{\chi^1_t-\chi^2_t}^2\leq \int_\mathbb{R}\int_{\R^d}\abs{\overline{\chi}(\rho^1_0)-\overline{\chi}(\rho^2_0)}^2=\int_{\R^d}\abs{\rho^1_0-\rho^2_0},
\end{align*}
where here we are using the fact that $\rho_0^1-\rho_0^2\in L^1(\R^d)$ because $\rho_0^i\in L^1_\gamma(\R^d)$, which completes the proof of the first statement. The proof of the second statement is identical, with the exception that if $\rho_0^1=\rho_0^2$ then in the final step we have that, $\P$-a.s.\ for every $t\in[0,T]$,
\[\int_\mathbb{R}\int_{\R^d}\abs{\chi^1_s-\chi^2_s}^2 \leq \lim_{\beta\rightarrow 0}\lim_{M\rightarrow\infty}\lim_{R\rightarrow\infty}\lim_{\d\rightarrow 0}\lim_{\ve\rightarrow 0}\Big(-2I^{\textrm{err}}_t-2I^{\textrm{meas}}_t+I^{\textrm{mart}}_t+I^{\textrm{cut}}_t+I^{\textrm{cons}}_t\Big)=0,\]
where in both cases on the lefthand side of the inequality the limits $R,M\rightarrow\infty$ and $\beta\rightarrow 0$ are justified using the monotone convergence theorem.  We therefore conclude that, in the case $\rho^1_0=\rho^2_0\in L^1(\O;\Ent_{\Phi,\gamma}(\R^d))$, any two stochastic kinetic solutions $\P$-a.s.\ satisfy $\rho^1=\rho^2\in\Ent_{\Phi,\gamma}(\R^d)$ for every $t\in[0,T]$, which completes the proof.  \end{proof}

\subsection{Existence of renormalized kinetic solutions}\label{section_exist}

In this section, we will begin by establishing the existence of solutions to the regularized equation
\begin{equation}\label{ap_1}\partial_t\rho = \Delta\Phi(\rho)+\eta\Delta\rho - \nabla\cdot(\sigma(\rho)\dd\xi^\mathbf{a}+\nu(\rho))+\frac{\langle\xi^\mathbf{a}\rangle_1}{2}\nabla\cdot(\sigma'(\rho)^2\nabla\rho),\end{equation}
for smooth and bounded nonlinearities $\Phi$, $\sigma$, and $\nu$, for $\eta\in(0,1)$, for $\rho_0\in L^2(\O;(\Ent_{\Phi,\gamma}\cap L^2_\gamma)(\R^d))$, for noise satisfying Assumption~\ref{a_random}, and for coefficients satisfying the following assumption.

\begin{assumption}\label{as_compact}  Let $\Phi,\sigma\in\textrm{C}([0,\infty))\cap\textrm{C}^1_{\textrm{loc}}((0,\infty))$ and $\nu\in \C([0,\infty))^d\cap\C^1((0,\infty))^d$ satisfy $\Phi(0)=\sigma(0)=\nu(0)=0$ and $\Phi'(\xi)>0$ for every $\xi\in(0,\infty)$.  Assume furthermore that $\Phi$, $\sigma$, and $\nu$ satisfy the following six properties.

\begin{enumerate}[(i)]
\item There exists $p\in[1,\infty)$ and $c\in(0,\infty)$ such that
\begin{equation}\label{assume_f20}\Phi(\xi)\leq c(1+\xi^p)\;\;\textrm{for every}\;\;\xi\in[0,\infty).\end{equation}
\item We have that $\log(\Phi)$ is locally integrable on $[0,\infty)$.
\item There exists $m\in[1,\infty)$ and $c_1,c_2\in[0,\infty)$ such that, for every $\xi\in[0,\infty)$,
\[\abs{\Phi^\frac{1}{2}(\xi)-\Phi^\frac{1}{2}((\nicefrac{3\gamma}{2}\wedge \xi)\vee \nicefrac{\gamma}{2})}\leq c_1\abs{\xi-((\nicefrac{3\gamma}{2}\wedge \xi)\vee\nicefrac{\gamma}{2})}^\frac{m}{2}+c_2\mathbf{1}_{\{\abs{\xi-\gamma}\geq\nicefrac{\gamma}{2}\}}.\]
\item Either there exists $c\in(0,\infty)$ and $\theta\in[0,\nicefrac{1}{2}]$ such that
\[\frac{\Phi'(\xi)}{\Phi^\frac{1}{2}(\xi)}\leq c\xi^{-\theta}\;\;\textrm{for every}\;\;\xi\in(0,\infty),\]
or there exists $c\in(0,\infty)$ and $q\in[1,\infty)$ such that
\[\abs{\xi-\xi'}^q\leq c\abs{\Phi^\frac{1}{2}(\xi)-\Phi^\frac{1}{2}(\xi')}^2\;\;\textrm{for every}\;\;\xi,\xi'\in[0,\infty).\]
\item There exists $c\in(0,\infty)$ such that
\[\abs{\sigma(\xi)}\leq c\Phi^\frac{1}{2}(\xi)\;\;\textrm{and}\;\;\abs{\nu(\xi)}\leq c\Phi(\xi)\;\;\textrm{for every}\;\;\xi\in[0,\infty).\]
\item There exists $c\in(0,\infty)$ and $q\in[0,2)$ such that
\[\Phi'(\xi)\leq c\big(1+(\Phi^\frac{1}{2}(\xi))^q\big)\;\;\textrm{for every $\xi\in[0,\infty)$}.\]
\end{enumerate}
\end{assumption}

\begin{definition}\label{def_sol_smooth}  Let $\xi^\mathbf{a}$ satisfy Assumption~\ref{a_random}, let $\Phi$, $\sigma$, and $\nu$ be smooth, $\C^2$-bounded, and satisfy Assumption~\ref{as_compact}, and let $\rho_0\in L^2(\O;(L^2_\gamma\cap\Ent_{\Phi,\gamma})(\R^d))$ be $\F_0$-measurable.  A solution of \eqref{ap_1} is an $\F_t$-adapted, continuous $L^2_\gamma$-valued process $\rho\in L^2(\O;L^2([0,T];H^1_\gamma(\R^d)))$ that satisfies, for every $t\in[0,T]$ and $\psi\in\C^\infty_c(\R^d)$,
\begin{align*}
\int_{\R^d}\rho(x,s)\psi(x)\Big|_{s=0}^{s=t} & =-\int_0^t\int_{\R^d}\Phi'(\rho)\nabla\phi\cdot\nabla\psi-\eta\int_0^t\int_{\R^d}\nabla\rho\cdot\nabla\psi+\int_0^t\int_{\R^d}\sigma(\rho)\nabla\psi\cdot\dd\xi^\mathbf{a}
\\ & \quad +\int_0^t\int_{\R^d}\nu(\rho)\cdot\nabla\psi-\frac{\langle\xi^\mathbf{a}\rangle_1}{2}\int_0^t\int_{\R^d}\sigma'(\rho)^2\nabla\rho\cdot\nabla\psi.
\end{align*}
\end{definition}

\begin{prop}\label{prop_entropy}  Let $\xi^\mathbf{a}$ satisfy Assumption~\ref{a_random}, let $\Phi$, $\sigma$, and $\nu$ be smooth, $\C^2$-bounded, and satisfy Assumption~\ref{as_compact}, and let $\rho_0\in L^2(\O;(L^2_\gamma\cap\Ent_{\Phi,\gamma})(\R^d))$ be $\F_0$-measurable.  Then, there exists a unique solution of \eqref{ap_1} in the sense of Definition~\ref{def_sol_smooth}.  \end{prop}

\begin{proof}  The proof is based on standard techniques and details can be found in \cite[Section~5]{FG21}.  \end{proof}

We will now establish stable a priori estimates for the solutions of \eqref{ap_1} with smooth coefficients and $\eta\in(0,1)$.  Based on these estimates, we will deduce the existence of solutions to \eqref{i_1} with coefficients satisfying Assumption~\ref{as_compact}.

\begin{prop}  Let $\xi^\mathbf{a}$ satisfy Assumption~\ref{a_random}, let $\Phi$, $\sigma$, and $\nu$ be smooth, $\C^2$-bounded, and satisfy Assumption~\ref{as_compact}, let $\rho_0\in L^2(\O;(L^2_\gamma\cap\Ent_{\Phi,\gamma})(\R^d))$ be $\F_0$-measurable, and let $\rho$ be a solution of \eqref{ap_1} in the sense of Definition~\ref{def_sol_smooth} with initial data $\rho_0$.  Then, there exists $c\in(0,\infty)$ such that, if $d=1$,
\begin{align*}
& \E\Big[\max_{t\in[0,T]}\int_{\R^d}\Psi_{\Phi,\gamma}(\rho(x,t)) +  \int_0^T\int_{\R^d}\abs{\nabla\Phi^\frac{1}{2}(\rho)}^2\big]
\\ &   \leq c\E \Big[\int_{\R^d}\Psi_{\Phi,\gamma}(\rho_0)+1+\langle\xi^\mathbf{a}\rangle_1+\norm{\langle \nabla\cdot\xi^\mathbf{a}\rangle_1 }_{L^1(\R^d)}+\norm{\langle \nabla\cdot\xi^\mathbf{a}\rangle_1 }_{L^1(\R^d)}^\frac{2}{2-q}\Big],
\end{align*}
and, if $d=2$, there exist $c,\theta\in(0,\infty)$ such that
\begin{align*}
& \E\Big[\max_{t\in[0,T]}\int_{\R^d}\Psi_{\Phi,\gamma}(\rho(x,t)) +  \int_0^t\int_{\R^d}\abs{\nabla\Phi^\frac{1}{2}(\rho)}^2\Big]
\\ \nonumber &   \leq  c\E\Big[\int_{\R^d}\Psi_{\Phi,\gamma}(\rho_0)+1+\langle\xi^\mathbf{a}\rangle_1^\frac{1}{2}A^{-\frac{1}{2}}+\langle\xi^\mathbf{a}\rangle_1A^{-1}+\norm{\langle \nabla\cdot\xi^\mathbf{a}\rangle_1 }_{L^1(\R^d)}+\norm{\langle \nabla\cdot\xi^\mathbf{a}\rangle_1 }_{L^\infty(\R^d)}^\theta\Big],
\end{align*}
and, if $d\geq 3$, there exist $c,\theta\in(0,\infty)$ such that
\begin{align*}
& \E\Big[\max_{t\in[0,T]}\int_{\R^d}\Psi_{\Phi,\gamma}(\rho(x,t)) +  \int_0^t\int_{\R^d}\abs{\nabla\Phi^\frac{1}{2}(\rho)}^2\big]
\\ &   \leq c\E \Big[\int_{\R^d}\Psi_{\Phi,\gamma}(\rho_0)+1+\langle\xi^\mathbf{a}\rangle_1A^{1-\frac{d+2}{2}}+\norm{\langle \nabla\cdot\xi^\mathbf{a}\rangle_1 }_{L^1(\R^d)}+\norm{\langle \nabla\cdot\xi^\mathbf{a}\rangle_1 }^\frac{2}{2-q}_{L^\frac{2}{2-q}(\R^d)}\Big].
\end{align*}
\end{prop}

\begin{proof}  Let $\Psi_{\Phi,\gamma,\d}$ be the unique function satisfying $\Psi_{\Phi,\gamma,\d}(\gamma)=0$ and $\Psi_{\Phi,\gamma,\d}'(\xi) = \log(\nicefrac{(\Phi(\xi)+\d)}{\Phi(\gamma)})$.  It is then a consequence of localization argument and It\^o's formula (see, for example, Krylov~\cite{Kry2013}) that, for every $t\in[0,T]$,
\begin{align*}
& \int_{\R^d}\Psi_{\Phi,\gamma,\d}(\rho(x,t)) = \int_{\R^d}\Psi_{\Phi,\gamma,\d}(\rho_0(x))-\int_0^t\int_{\R^d}\frac{\Phi'(\rho)^2}{\d+\Phi(\rho)}\abs{\nabla\rho}^2-\eta\int_0^t\int_{\R^d}\frac{\Phi'(\rho)}{\d+\Phi(\rho)}\abs{\nabla\rho}^2
\\ & \quad + \int_0^t\int_{\R^d}\frac{\sigma(\rho)\Phi'(\rho)}{\d+\Phi(\rho)}\nabla\rho\cdot\dd\xi^{\mathbf{a}}+\int_0^t\int_{\R^d}\nu(\rho)\cdot\frac{\Phi'(\rho)}{\d+\Phi(\rho)}\nabla\rho +\frac{1}{2}\int_0^t\int_{\R^d}\langle \nabla\cdot\xi^\mathbf{a}\rangle_1 \frac{\sigma(\rho)^2\Phi'(\rho)}{\d+\Phi(\rho)}.
\end{align*}
Since for $\Theta_{\nu,\d}=(\Theta_{\nu,\d,i})_{i=1}^d$ satisfying $\Theta_{\nu,\d,i}(0)=0$ and $\Theta_{\nu,\d,i}'(\xi) = \frac{\nu_{n,i}(\xi)\Phi'(\xi)}{\d+\Phi(\xi)}$ we have that
\[\int_0^t\int_{\R^d}\nu(\rho)\cdot\frac{\Phi'(\rho)}{\d+\Phi(\rho)}\nabla\rho = \int_0^T\int_{\R^d}\nabla\cdot\Theta_{\nu,\d}(\rho) = 0,\]
it follows $\P$-a.s.\ from the assumption $\abs{\sigma(\xi)}\leq c\Phi^\frac{1}{2}(\xi)$ that, for $c\in(0,\infty)$ independent of $\d\in(0,1)$,
\begin{align}\label{4_40}
& \max_{t\in[0,T]}\int_{\R^d}\Psi_{\Phi,\gamma,\d}(\rho(x,t)) +  \int_0^T\int_{\R^d}\frac{\Phi'(\rho)^2}{\d+\Phi(\rho)}\abs{\nabla\rho}^2
\\ \nonumber & \leq  \int_{\R^d}\Psi_{\Phi,\gamma,\d}(\rho_0(x)) +\max_{t\in[0,T]}\abs{\int_0^t\int_{\R^d}\frac{\sigma(\rho)\Phi'(\rho)}{\d+\Phi(\rho)}\nabla\rho\cdot\dd\xi^{\mathbf{a}}}+c\int_0^T\int_{\R^d} \langle \nabla\cdot\xi^\mathbf{a}\rangle_1 \Phi'(\rho).
\end{align}
We will now treat the stochastic integral, where for $\Phi^{\frac{1}{2},\d}$ satisfying $\Phi^{\frac{1}{2},\d}(0)=0$ and $(\Phi^{\frac{1}{2},\d})'(\xi) = \frac{\sigma(\xi)\Phi'(\xi)}{\d+\Phi(\xi)}$ we have that, for every $t\in[0,T]$,
\[\abs{\int_0^t\int_{\R^d}\frac{\sigma(\rho)\Phi'(\rho)}{\d+\Phi(\rho)}\nabla\rho\cdot\dd\xi^{\mathbf{a}}} =\abs{\int_0^t\int_{\R^d}\nabla\Phi^{\frac{1}{2},\d}(\rho)\cdot \dd \xi^{\mathbf{a}}}.\]
It then follows from the boundedness and regularity of $\Phi$, the bound $\abs{\sigma}\leq c\Phi^\frac{1}{2}$, the definition of $\xi^\mathbf{a}$, and the Burkholder--Davis--Gundy inequality that, for some $c\in(0,\infty)$,
\begin{align}\label{4_30}
& \E\big[\max_{t\in[0,T]}\abs{\int_0^t\int_{\R^d}\nabla\Phi^{\frac{1}{2},\d}(\rho)\cdot\dd\xi^{\mathbf{a}}}\big]
\\ \nonumber & \leq c\E\big[\big(\int_0^T\sum_{k=1}^\infty\sum_{i=1}^d\big(\int_{\R^d}\partial_i\Phi^{\frac{1}{2},\d}(\rho)\exp(-\frac{A}{2}\abs{x}^2)(f_k*\kappa^\a)\big)^2\big)^\frac{1}{2}\big]
\\ \nonumber &\quad + c\langle\xi^\mathbf{a}\rangle_1^\frac{1}{2}\E\Big[\Big(\int_0^T\sum_{i=1}^d\Big(\int_{\R^d}(\Phi^{\frac{1}{2},\d}(\rho)-\Phi^{\frac{1}{2},\d}(\gamma))\partial_i\sqrt{1-\exp(-A\abs{x}^2)}\Big)^2\Big)^\frac{1}{2}\Big].
\end{align}
For the first term on the righthand side of \eqref{4_30}, using that the $f_k$ are an orthonormal $L^2(\R^d)$ basis, H\"older's inequality, Young's inequality, and that, for every $\d\in(0,1)$,
\begin{equation}\label{4_39}\int_0^T\norm{\nabla\Phi^{\frac{1}{2},\d}(\rho)}^2_{L^2(\R^d)}\leq \int_0^T\int_{\R^d}\frac{\Phi'(\rho)^2}{\d+\Phi(\rho)}\abs{\nabla\rho}^2,\end{equation}
we have that
\begin{align*}
& \E\big[\big(\int_0^T\sum_{k=1}^\infty\sum_{i=1}^d\big(\int_{\R^d}\partial_i\Phi^{\frac{1}{2},\d}(\rho)\exp(-\frac{A}{2}\abs{x}^2)(f_k*\kappa^\a)\big)^2\big)^\frac{1}{2}\big]
\\ & \leq c\E\big[\big(\int_0^T\norm{(\nabla\Phi^{\frac{1}{2},\d}(\rho)\exp(-\frac{A}{2}\abs{x}^2))*\kappa^\a}^2_{L^2(\R^d)}\big)^\frac{1}{2}\big]
\\ &  \leq c\E\big[\int_0^T\norm{\nabla\Phi^{\frac{1}{2},\d}(\rho)}^2_{L^2(\R^d)}\big]^\frac{1}{2} \leq c\E\big[\int_0^T\int_{\R^d}\frac{\Phi'(\rho)^2}{\d+\Phi(\rho)}\abs{\nabla\rho}^2\big]^\frac{1}{2}.
\end{align*}
The second term on the righthand side of \eqref{4_30} will be treated in cases depending on the dimension using the Sobolev embedding theorem.  We first observe that, based on an explicit computation, for every $p\in[1,\infty]$ and $d\in\N$ there exists $c_{p,d}\in(0,\infty)$ such that
\[\norm{\nabla(\sqrt{1-\exp(-A\abs{x}^2)})}_{L^p(\R^d)^d}\leq c_{p,d}A^{\frac{1}{2}-\frac{d}{2p}}.\]
We now proceed by cases.  If $d=1$,  using the Sobolev embedding theorem, H\"older's inequality, Young's inequality, and \eqref{4_39}, for every $\ve\in(0,1)$ there exists a constant $c\in(0,\infty)$ depending on $\ve$ such that
\begin{align}\label{4_41}
&\langle\xi^\mathbf{a}\rangle_1^\frac{1}{2} \E\Big[\Big(\int_0^T\sum_{i=1}^d\Big(\int_{\R^d}(\Phi^{\frac{1}{2},\d}(\rho)-\Phi^{\frac{1}{2},\d}(\gamma))\partial_i\sqrt{1-\exp(-A\abs{x}^2)}\Big)^2\Big)^\frac{1}{2}\Big]
\\ \nonumber & \leq c\langle\xi^\mathbf{a}\rangle_1^\frac{1}{2}\E\big[\big(\int_0^T\norm{\Phi^{\frac{1}{2},\d}(\rho)-\Phi^{\frac{1}{2},\d}(\gamma)}^2_{L^\infty(\R^d)}\norm{\nabla(\sqrt{1-\exp(-A\abs{x}^2)})}^2_{L^1(\R^d)^d}\big)^\frac{1}{2}\big]
\\ \nonumber & \leq c\langle\xi^\mathbf{a}\rangle_1+\ve\E\big[\int_0^T\int_{\R^d}\frac{\Phi'(\rho)^2}{\d+\Phi(\rho)}\abs{\nabla\rho}^2\big].
\end{align}
The case $d=2$ is the most complicated, due to the criticality of the Sobolev embedding.  In this case, we write
\begin{equation}\label{4_50}(\Phi^{\frac{1}{2},\d}(\rho)-\Phi^{\frac{1}{2},\d}(\gamma)) = \Phi^{\frac{1}{2},\d}((\nicefrac{3\gamma}{2}\wedge \rho)\vee \nicefrac{\gamma}{2})+(\Phi^{\frac{1}{2},\d}(\rho)-\Phi^{\frac{1}{2},\d}((\nicefrac{3\gamma}{2}\wedge \rho)\vee \nicefrac{\gamma}{2})).\end{equation}
The first term is bounded, and can be treated similarly to the case $d=1$.  For the second term, we potentially use the less standard Sobolev spaces for $p\in(0,\infty]$ as opposed to $p\in[1,\infty]$.  It then follows from the growth assumption \eqref{assume_f20} on $\Phi^\frac{1}{2}$ and \eqref{4_39} that, for every $t\in[0,T]$, for some $c\in(0,\infty)$ depending on $\gamma$,
\[\norm{\Phi^{\frac{1}{2},\d}(\rho)-\Phi^{\frac{1}{2},\d}((\nicefrac{3\gamma}{2}\wedge \rho)\vee \nicefrac{\gamma}{2})}^\frac{2}{m}_{L^\frac{2}{m}(\R^d)}\leq \norm{(\rho-\gamma)\mathbf{1}_{\{\abs{\rho-\gamma}\geq \nicefrac{\gamma}{2}\}}}_{L^1(\R^d)}\leq c\norm{\Psi_{\Phi,\gamma,\d}(\rho)}_{L^1(\R^d)},\]
where the final inequality is a consequence of \eqref{assume_f30}.  We then have by interpolation---which remains a consequence of H\"older's inequality for the less standard Sobolev spaces $p\in[0,1)$---and the Sobolev embedding theorem that for every $p\in(\frac{2}{m}\vee 1,\infty)$ there exists $\theta\in(0,1)$ such that, for $c\in(0,\infty)$ depending on $\gamma$ and the Sobolev embedding theorem,
\begin{align*}
& \norm{\Phi^{\frac{1}{2},\d}(\rho)-\Phi^{\frac{1}{2},\d}((\nicefrac{3\gamma}{2}\wedge \rho)\vee \nicefrac{\gamma}{2})}_{L^p(\R^d)}
\\ & \leq c\norm{\Psi_{\Phi,\gamma,\d}(\rho)}_{L^1(\R^d)}^{\frac{m\theta}{2}}\norm{\Phi^{\frac{1}{2},\d}(\rho)-\Phi^{\frac{1}{2},\d}((\nicefrac{3\gamma}{2}\wedge \rho)\vee \nicefrac{\gamma}{2})}^{1-\theta}_{BMO(\R^d)}
\\ & \leq c\norm{\Psi_{\Phi,\gamma,\d}(\rho)}_{L^1(\R^d)}^{\frac{m\theta}{2}}\norm{\nabla\Phi^{\frac{1}{2},\d}(\rho)}^{1-\theta}_{L^2(\R^d)}.
\end{align*}
Using the computation for the $d=1$ case to handle the bounded part of \eqref{4_50}, we have using H\"older's inequality that, for $c\in(0,\infty)$ depending on $\gamma$, $p$, and $d$,
\begin{align}\label{4_051}
& \langle\xi^\mathbf{a}\rangle_1^\frac{1}{2}\E\Big[\Big(\int_0^T\sum_{i=1}^d\Big(\int_{\R^d}(\Phi^{\frac{1}{2},\d}(\rho)-\Phi^{\frac{1}{2},\d}(\gamma))\partial_i\sqrt{1-\exp(-A\abs{x}^2)}\Big)^2\Big)^\frac{1}{2}\Big]
\\ \nonumber & \leq c\langle \xi^\mathbf{a}\rangle_1^\frac{1}{2}A^{-\frac{1}{2}} +c\langle\xi^\mathbf{a}\rangle_1^\frac{1}{2}T^\frac{\theta}{2}A^{\frac{1}{2}-\frac{p-1}{p}}\E\big[\max_{t\in[0,T]}\norm{\Psi_{\Phi,\gamma,\d}(\rho)}_{L^1(\R^d)}^{{\frac{m\theta}{2}}}\big(\int_0^T\norm{\nabla\Phi^{\frac{1}{2},\d}(\rho)}^2_{L^2(\R^d)}\big)^\frac{1-\theta}{2}\big].
\end{align}
Since $\theta\rightarrow 0$ as $p\rightarrow\infty$, for all $p$ sufficiently large such that $\theta m\leq 1$ we have using H\"older's inequality, Young's inequality, \eqref{4_39}, and \eqref{4_051} that, for every $\ve\in(0,1)$ there exists $c\in(0,\infty)$ depending on $T$, $m$ and $\ve$ such that, for $\tilde{\theta} = 2((1-m\theta)+\theta)^{-1}$,
\begin{align}\label{4_51}
& \E\Big[\Big(\int_0^T\sum_{i=1}^d\Big(\int_{\R^d}(\Phi^{\frac{1}{2},\d}(\rho)-\Phi^{\frac{1}{2},\d}(\gamma))\partial_i\sqrt{1-\exp(-A\abs{x}^2)}\Big)^2\Big)^\frac{1}{2}\Big]
\\ \nonumber & \leq c(\langle\xi^\mathbf{a}\rangle_1^\frac{1}{2}A^{-\frac{1}{2}}+\langle\xi^\mathbf{a}\rangle_1^{\frac{\tilde{\theta}}{2}}A^{\tilde{\theta}(\frac{1}{2}-\frac{p-1}{p})})+\ve\E\big[\max_{t\in[0,T]}\norm{\Psi_{\Phi,\gamma,\d}(\rho)}_{L^1(\R^d)}\big]+\ve\E\big[\int_0^T\int_{\R^d}\frac{\Phi'(\rho)^2}{\d+\Phi(\rho)}\abs{\nabla\rho}^2\big].
\end{align}
It remains to treat the case $d\geq 3$.  In this case, it follows from the Sobolev embedding theorem and H\"older's inequality that for every $\ve\in(0,1)$ there exists $c\in(0,\infty)$ depending on $\ve$ and $d$ such that, for $\nicefrac{1}{2_*}=\nicefrac{1}{2}-\nicefrac{1}{d}$,
\begin{align}\label{4_52}
& \langle\xi^\mathbf{a}\rangle_1^\frac{1}{2}\E\Big[\Big(\int_0^T\sum_{i=1}^d\Big(\int_{\R^d}(\Phi^{\frac{1}{2},\d}(\rho)-\Phi^{\frac{1}{2},\d}(\gamma))\partial_i\sqrt{1-\exp(-A\abs{x}^2)}\Big)^2\Big)^\frac{1}{2}\Big]
 \\ \nonumber & \leq c \langle\xi^\mathbf{a}\rangle^\frac{1}{2}_1A^{\frac{1}{2}-\frac{d+2}{4}}\E\Big[\int_0^T\Big(\int_{\R^d}\big(\Phi^{\frac{1}{2},\d}(\rho)-\Phi^{\frac{1}{2},\d}(\gamma)\big)^{2_*}\Big)^\frac{2}{2_*}\Big]^\frac{1}{2}
 \\ \nonumber & \leq c\langle\xi^\mathbf{a}\rangle_1A^{1-\frac{d+2}{2}}+\ve\E\big[\int_0^T\int_{\R^d}\abs{\nabla\Phi^{\frac{1}{2},\d}(\rho)}^2\big].
 \end{align}
Returning to \eqref{4_40} and applying \eqref{4_41}, \eqref{4_51}, and \eqref{4_52} for $\ve\in(0,1)$ sufficiently small, we conclude that for all $A\in(0,1]$ there exists $c\in(0,\infty)$ such that, after passing $\d\rightarrow 0$ using the monotone convergence theorem, for $d=1$ and $d\geq 3$,
\begin{align}\label{4_53}
& \E\Big[\max_{t\in[0,T]}\int_{\R^d}\Psi_{\Phi,\gamma}(\rho(x,t)) +  \int_0^T\int_{\R^d}\abs{\nabla\Phi^\frac{1}{2}(\rho)}^2\big]
\\ \nonumber &   \leq c\E \Big[\int_{\R^d}\Psi_{\Phi,\gamma}(\rho_0)+\langle\xi^\mathbf{a}\rangle_1A^{1-\frac{d+2}{2}}+\int_0^T\int_{\R^d} \langle \nabla\cdot\xi^\mathbf{a}\rangle_1 \Phi'(\rho)\Big],
\end{align}
and if $d=2$ there exists $c\in(0,\infty)$ such that
\begin{align}\label{4_54}
& \E\Big[\max_{t\in[0,T]}\int_{\R^d}\Psi_{\Phi,\gamma}(\rho(x,t)) +  \int_0^T\int_{\R^d}\abs{\nabla\Phi^\frac{1}{2}(\rho)}^2\Big]
\\ \nonumber &   \leq  c\E\Big[\int_{\R^d}\Psi_{\Phi,\gamma}(\rho_0)+\langle \xi^\mathbf{a}\rangle_1^\frac{1}{2}A^{-\frac{1}{2}}+\langle\xi^\mathbf{a}\rangle_1A^{-1}+\int_0^T\int_{\R^d} \langle \nabla\cdot\xi^\mathbf{a}\rangle_1 \Phi'(\rho)\Big].
\end{align}
For the final term on the righthand side of \eqref{4_40}, we have using the bound on $\Phi'$ that, for $c\in(0,\infty)$ depending on $\gamma$ and $q$,
\[\int_0^T\int_{\R^d} \langle \nabla\cdot\xi^\mathbf{a}\rangle_1 \Phi'(\rho)\leq \int_0^T\int_{\R^d}\langle \nabla\cdot\xi^\mathbf{a}\rangle_1 (\Phi^\frac{1}{2}(\rho)-\Phi^\frac{1}{2}(\gamma))^q+c\int_0^T\int_{\R^d} \langle \nabla\cdot\xi^\mathbf{a}\rangle_1 .\]
We treat the first term on the righthand side analogously to the above arguments using the Sobolev embedding theorem.  If $d=1$, using $q\in[0,2)$ and \eqref{4_39}, this yields that for every $\ve\in(0,1)$ there exists $c\in(0,\infty)$ depending on $\ve$ such that
\[\int_0^T\int_{\R^d} \langle \nabla\cdot\xi^\mathbf{a}\rangle_1 \Phi'(\rho)\leq cT\norm{\langle \nabla\cdot\xi^\mathbf{a}\rangle_1 }_{L^1(\R^d)}+c\norm{\langle \nabla\cdot\xi^\mathbf{a}\rangle_1 }_{L^1(\R^d)}^\frac{2}{2-q}+\ve\int_0^T\int_{\R^d}\abs{\nabla\Phi^\frac{1}{2}(\rho)}^2.\]
If $d=2$, we obtain that for every $\ve\in(0,1)$ there exists $c\in(0,\infty)$ depending on $\ve$ and $\theta\in(0,\infty)$ such that
\begin{align*}
& \int_0^T\int_{\R^d} \langle \nabla\cdot\xi^\mathbf{a}\rangle_1 \Phi'(\rho)
\\ & \leq cT\norm{\langle \nabla\cdot\xi^\mathbf{a}\rangle_1 }_{L^1(\R^d)}+c\norm{\langle \nabla\cdot\xi^\mathbf{a}\rangle_1 }_{L^\infty(\R^d)}^\theta
\\ &\quad +\ve\E\big[\max_{t\in[0,T]}\norm{\Psi_{\Phi,\gamma,\d}(\rho)}_{L^1(\R^d)}\big]+\ve\E\big[\int_0^T\int_{\R^d}\frac{\Phi'(\rho)^2}{\d+\Phi(\rho)}\abs{\nabla\rho}^2\big].
\end{align*}
If $d\geq 3$, since $q\in[0,2)$, it follows from H\"older's inequality and the Sobolev inequality that for every $\ve\in(0,1)$ there exists $c\in(0,\infty)$ depending on $\ve$ and $T$ such that, for $\nicefrac{1}{2_*}=\nicefrac{1}{2}-\nicefrac{1}{d}$,
\begin{align*}
\int_0^T\int_{\R^d} \langle \nabla\cdot\xi^\mathbf{a}\rangle_1 \Phi'(\rho)& \leq cT\norm{\langle \nabla\cdot\xi^\mathbf{a}\rangle_1 }_{L^1(\R^d)}
\\ & \quad + \norm{\langle \nabla\cdot\xi^\mathbf{a}\rangle_1 }_{L^\frac{2}{2-q}(\R^d)}\big(\int_0^T\big(\int_{\R^d}(\Phi^\frac{1}{2}(\rho)-\Phi^\frac{1}{2}(\gamma))^{2_*}\big)^\frac{2}{2_*}\big)^\frac{q}{2}
\\ & \leq cT\norm{\langle \nabla\cdot\xi^\mathbf{a}\rangle_1 }_{L^1(\R^d)}+c\norm{\langle \nabla\cdot\xi^\mathbf{a}\rangle_1 }^\frac{2}{2-q}_{L^\frac{2}{2-q}(\R^d)}+\ve\int_0^T\int_{\R^d}\abs{\nabla\Phi^\frac{1}{2}(\rho)}^2.
\end{align*}
After applying these inequalities to \eqref{4_53} and \eqref{4_54}, we complete the proof. \end{proof}

\begin{prop}\label{prop_measure} Let $\xi^\mathbf{a}$ satisfy Assumption~\ref{a_random}, let $\Phi$, $\sigma$, and $\nu$ be smooth, $\C^2$-bounded, and satisfy Assumption~\ref{as_compact}, let $\rho_0\in L^1(\O;\Ent_{\Phi,\gamma}(\R^d))$, and let $\rho$ be a solution of \eqref{ap_1} in the sense of Definition~\ref{def_sol_smooth} with initial data $\rho_0$.  Then,
\[\E\big[\int_0^T\int_{\R^d}\mathbf{1}_{[M,M+1]}(\rho)\Phi'(\rho)\abs{\nabla\rho}^2\big]\leq \E\big[\int_{\R^d}(\rho-M)_++\int_0^T\int_{\R^d}\langle \nabla\cdot\xi^\mathbf{a}\rangle_1 \mathbf{1}_{[M,M+1]}(\rho)\big].\]
\end{prop}

\begin{proof}  Let $\Psi_M$ be the unique function satisfying $\Psi_M'(0)=\Psi_M(0)=0$ and $\Psi''_M(\xi) = \mathbf{1}_{[M,M+1]}(\xi)$.  It then follows from an approximation argument, a localization argument, and It\^o's formula (see, for example, \cite{Kry2013}) using the methods of Proposition~\ref{prop_entropy} that, $\P$-a.s.\ for every $t\in[0,T]$,
\[\E\big[\int_{\R^d}\Psi_M(\rho(x,t))+\int_0^t\int_{\R^d}\mathbf{1}_{[M,M+1]}(\rho)\Phi'(\rho)\abs{\nabla\rho}^2\big] = \E\big[\int_{\R^d}\Psi_M(\rho_0)+\int_0^t\int_{\R^d}\langle \nabla\cdot\xi^\mathbf{a}\rangle_1 \mathbf{1}_{[M,M+1]}(\rho)\big].\]
We therefore have using the definition of $\Psi_M$ that
\[\E\big[\int_0^T\int_{\R^d}\mathbf{1}_{[M,M+1]}(\rho)\Phi'(\rho)\abs{\nabla\rho}^2\big]\leq \E\big[\int_{\R^d}(\rho-M)_++\int_0^T\int_{\R^d}\langle \nabla\cdot\xi^\mathbf{a}\rangle_1 \mathbf{1}_{[M,M+1]}(\rho)\big],\]
which completes the proof.  \end{proof}

\begin{thm}\label{thm_exist} Let $\xi^\mathbf{a}$ satisfy Assumption~\ref{a_random}, let $\Phi$, $\sigma$, and $\nu$ satisfy Assumptions~\ref{assume_1} and \ref{as_compact}, and let $\rho_0\in L^1(\O;\Ent_{\Phi,\gamma}(\R^d))$ be $\F_0$-measurable.  Then there exists a solution of \eqref{i_1} in the sense of Definition~\ref{d_sol}.  Furthermore, the solution satisfies the estimates of Propositions~\ref{prop_entropy} and \ref{prop_measure}.\end{thm}

\begin{proof}  The proof is a consequence of the estimates in Proposition~\ref{prop_entropy} and Proposition~\ref{prop_measure} and the methods of \cite[Theorem~5.25]{FG21}.  Here Assumption~\ref{assume_1} is used to guarantee the uniqueness of the solution, which is used to obtain a probabilistically strong solution.  \end{proof}

\section{The large deviations principle}\label{sec_ldp}

In this section, we will prove that, along appropriate scaling limits, the solutions
\begin{equation}\label{ldp_eq}\partial_t\rho^\ve = \Delta\Phi(\rho^\ve)-\nabla\cdot(\sqrt{\ve}\sigma(\rho^\ve)\dd\xi^{\mathbf{a}^\ve})+\frac{\ve \langle\xi^{\mathbf{a}^\ve}\rangle_1}{2}\nabla\cdot(\sigma'(\rho^\ve)^2\nabla\rho^\ve),\end{equation}
satisfy a large deviations principle in the strong topology of $L^1([0,T];L^1_{\textrm{loc}}(\R^d))$ with rate function
\begin{equation}\label{rate_function}I_{\rho_0}(\rho) = \frac{1}{2}\inf\{\norm{g}^2_{L^2(\R^d)^d}\colon \partial_t\rho = \Delta\Phi(\rho)-\nabla\cdot(\sigma(\rho)g)\;\;\textrm{with}\;\;\rho(\cdot,0)=\rho_0\;\;\textrm{in}\;\;\R^d\times[0,T]\}.\end{equation}
The skeleton equation appearing in the rate function is understood in terms of Definition~\ref{def_skel} below.

The proof is based on the weak approach to large deviations established, for example, in \cite{{BudDupMar2008}}.  For this, it is necessary to consider the following controlled SPDE.  Let $\mathbf{g}=(g_k,\tilde{g_k})_{k\in\N}\in L^2([0,T];\ell^2(\N)^d)^2$ be an arbitrary control and for the orthonormal $L^2(\R^d)$-basis $\{f_k\}_{k\in\N}$ in Assumption~\ref{a_random} let $g=\sum_{k=1}^\infty g_k(t)f_k(x)\in L^2(\R^d\times[0,T])^d$ and let $\tilde{g}=(\tilde{g}_k)_{k\in\N}$.  The controlled SPDE defined by $\mathbf{g}$ is
\begin{align}\label{ldp_1}
\partial_t\rho^\ve &= \Delta\Phi(\rho^\ve)-\nabla\cdot(\sqrt{\ve}\sigma(\rho^\ve)\dd\xi^{\mathbf{a}^\ve})+\frac{\ve \langle\xi^{\mathbf{a}^\ve}\rangle_1}{2}\nabla\cdot(\sigma'(\rho^\ve)^2\nabla\rho^\ve)
\\ \nonumber & \quad -\nabla\cdot(\sigma(\rho^\ve)\sqrt{\exp(-A^\ve\abs{x}^2)}(g*\kappa^{\a^\ve})+\sigma(\rho^\ve)\sqrt{1-\exp(-A^\ve\abs{x}^2)}\langle a^\ve,\tilde{g}\rangle_{\ell^2,t}).
\end{align}
The following proposition summarizes the well-posedness of \eqref{ldp_1}, where solutions are understood analogously to Definition~\ref{d_sol}.

\begin{prop}\label{prop_control}  Let $\ve\in(0,1)$, let $\xi^{\mathbf{a}^\ve}$ satisfy Assumption~\ref{a_random}, let $\Phi$ and $\sigma$ satisfy Assumptions~\ref{assume_1} and \ref{as_compact}, and let $\rho_0\in L^1(\O;\Ent_{\Phi,\gamma}(\R^d))$ be $\F_0$-measurable.  Then for every $\mathbf{g}\in L^2([0,T];\ell^2(\N)^d)^2$ there exists a unique solution of \eqref{ldp_1}.  Furthermore, the solutions satisfy the estimate of Propositions~\ref{prop_entropy} with the additional term, if $d=1$, for $c\in(0,\infty)$ depending on $T$, $d$, and $\gamma$,
\[c\int_0^T\int_{\R^d}\abs{g*\kappa^{\a^\ve}}^2+c\norm{a^\ve}^2_{\ell^\infty}\norm{\tilde{g}}^2_{\ell^2},\]
and, if $d\geq 2$, for $c\in(0,\infty)$ depending on $T$, $d$, and $\gamma$,
\[c\int_0^T\int_{\R^d}\abs{g*\kappa^{\a^\ve}}^2+c\norm{a^\ve}^2_{\ell^\infty}\norm{\tilde{g}}^2_{\ell^2}(A^\ve)^{1-\frac{d+2}{2}},\]
and satisfies the estimates of Proposition~\ref{prop_measure} with the extra term, for $c\in(0,\infty)$,
\[ c\int_0^T\int_{\R^d}\rho \abs{g*\kappa^{\a^\ve}}^2\mathbf{1}_{\{[M,M+1]}(\rho).\]
Furthermore, under Assumption~\ref{a_random} we have, for $c\in(0,\infty)$ independent of $\ve$,
\[\langle \sqrt{\ve}\xi^{\mathbf{a}^\ve}\rangle_1\leq c\ve (\a^\ve)^{-d},\]
and
\[\norm{\langle \sqrt{\ve}\nabla\cdot\xi^{\mathbf{a}^\ve}\rangle_1}_{L^1(\R^d)}\leq c\ve(\a^\ve)^{-d-2}(A^\ve)^{-\frac{d}{2}}\;\;\textrm{and}\;\;\norm{\langle \sqrt{\ve}\nabla\cdot\xi^{\mathbf{a}^\ve}\rangle_1}_{L^\infty(\R^d)}\leq c\ve(\a^\ve)^{-d-2}.\]
\end{prop}

\begin{proof}  The proof is a straightforward modification of Theorem~\ref{thm_unique}, Proposition~\ref{prop_entropy}, Proposition~\ref{prop_measure}, and Theorem~\ref{thm_exist}. \end{proof}

\begin{remark}  It follows from \cite[Lemma~7]{FehGes19} and the $L^2$-integrability of $g$ that, for every $\a\in[0,1)$,
\[\liminf_{M\rightarrow\infty}\int_0^T\int_{\R^d}\rho \abs{g*\kappa^\a}^2\mathbf{1}_{\{[M,M+1]}(\rho)=0.\]
The estimates of Proposition~\ref{prop_control} are therefore sufficient to establish the properties of Definition~\ref{d_sol}. \end{remark}

We will now establish the well-posedness of solutions to the skeleton equation defining the rate function \eqref{rate_function} under the following additional assumption.  We recall that the space $L^1_{\textrm{loc}}(\R^d)$ is equipped with the topology of local strong $L^1$-convergence.  That is, a sequence $\rho_n\rightarrow\rho$ in $L^1_{\textrm{loc}}(\R^d)$ if and only if $\rho_n\rightarrow\rho$ strongly in $L^1(B_R)$ for every $R>0$.

\begin{assumption}\label{as_equiv} Assume that $\Phi\in\textrm{C}([0,\infty))\cap\textrm{C}^1_{\textrm{loc}}((0,\infty))$ satisfies one of the following two conditions.
\begin{enumerate}
\item We have that $\Phi^\frac{1}{2}\colon[0,\infty)\rightarrow[0,\infty)$ is concave, and there exists $c\in(0,\infty)$ and $p\in[2,\infty)$ such that, for every $\xi\in(0,\infty)$,
\[\Big(\frac{\Phi^\frac{1}{2}(\xi)}{\Phi'(\xi)}\Big)^p\leq c(\xi+1).\]
\end{enumerate}
\begin{enumerate}
\item We have that $\Phi^\frac{1}{2}\colon[0,\infty)\rightarrow[0,\infty)$ is convex, there exists $c\in(0,\infty)$ such that
\[\sup_{\{\xi\geq 1\}}\abs{\frac{\Phi(\xi+1)}{\Phi(\xi)}}\leq c,\]
and for every $M\in(0,1)$ there exists $c\in(0,\infty)$ depending on $M$ such that
\[ \sup_{\{\xi\geq M\}}\abs{\frac{\Phi^\frac{1}{2}(\xi)}{\Phi'(\xi)}}\leq c.\]
\end{enumerate}
\end{assumption}

\begin{definition}\label{def_skel} Let $\Phi$ and $\sigma$ satisfy Assumptions~\ref{assume_1}, \ref{as_compact}, and \ref{as_equiv}, let $\gamma\in(0,\infty)$, let $\rho_0\in \Ent_{\Phi,\gamma}(\R^d)$, and let $g\in L^2(\R^d\times[0,T])^d$.  A solution of the skeleton equation
\begin{equation}\label{skel_eq} \partial_t\rho = \Delta\Phi(\rho)-\nabla\cdot(\sigma(\rho)g)\;\;\textrm{with}\;\;\rho(\cdot,0)=\rho_0,\end{equation}
is a continuous $L^1_{\textrm{loc}}(\R^d)$-valued function $\rho\in L^\infty([0,T];\Ent_{\Phi,\gamma}(\R^d))$ that satisfies the following two properties.
\begin{enumerate}
\item{The relative entropy estimate}:  we have that $\sup_{t\in[0,T]}\int_{\R^d}\Psi_{\Phi,\gamma}(\rho)+\int_0^T\int_{\R^d}\abs{\nabla\Phi^\frac{1}{2}(\rho)}^2<\infty$.
\item{The equation}:   for every $\psi\in\C^\infty_c(\R^d)$ and $t\in[0,T]$,
\[\int_{\R^d}\rho(x,t)\psi(x,t)  = \int_{\R^d}\rho_0(x)\psi(x) -\int_0^t\int_{\R^d}\nabla\Phi(\rho)\cdot\nabla\psi+\int_0^t\int_{\R^d}\sigma(\rho)g\cdot\nabla\psi.\]
\end{enumerate}
\end{definition}

\begin{thm}\label{thm_skel} Let $\Phi$ and $\sigma$ satisfy Assumptions~\ref{assume_1}, \ref{as_compact}, and \ref{as_equiv}, let $\gamma\in(0,\infty)$, let $\rho_0\in \Ent_{\Phi,\gamma}(\R^d)$, and let $g\in L^2(\R^d\times[0,T])^d$. Then there exists a unique solution $\rho$ of \eqref{skel_eq} in the sense of Definition~\ref{def_skel}.  Furthermore, for some $c\in(0,\infty)$,
\[\sup_{t\in[0,T]}\int_{\R^d}\Psi_{\Phi,\gamma}(\rho)+\int_0^T\int_{\R^d}\abs{\nabla\Phi^\frac{1}{2}(\rho)}^2\leq c\Big(\int_{\R^d}\Psi_{\Phi,\gamma}(\rho_0)+\int_0^T\int_{\R^d}\abs{g}^2\Big).\]
\end{thm}

\begin{proof}  The proof is a consequence of \cite[Sections~6, 7]{FehGesHey2024}.  \end{proof}

We are now prepared to state the large deviations principle.

\begin{thm}\label{prop_collapse}  Let $\Phi$ and $\sigma$ satisfy Assumptions~\ref{assume_1},  \ref{as_compact}, and \ref{as_equiv} and let $\a^\ve$ be a sequence that satisfies that $A^\ve,\a^\ve\rightarrow 0$ as $\ve\rightarrow 0$ and that, if $d\geq 2$,
\[\norm{a^\ve}^2_{\ell^\infty} (A^\ve)^{1-\frac{d+2}{2}}\rightarrow 0\;\;\textrm{and}\;\;\ve(\a^\ve)^{-d-2}(A^\ve)^{-\frac{d}{2}}\rightarrow 0,\]
and, if $d=1$,
\[\norm{a^\ve}_{\ell^\infty}\rightarrow 0\;\;\textrm{and}\;\;\ve(\a^\ve)^{-d-2}(A^\ve)^{-\frac{d}{2}}\rightarrow 0.\]
Then, the rate functions $I_{\rho_0}$  defined in \eqref{rate_function} are good rate functions with compact level sets on compact sets, and for every $\rho_0\in\Ent_{\Phi,\gamma}(\R^d)$ the solutions $\{\rho^\ve(\rho_0)\}_{\ve\in(0,1)}$ of \eqref{ldp_eq} satisfy a large deviations principle with rate function $I_{\rho_0}$ on $L^1([0,T];L^1_{\textrm{loc}}(\R^d))$.  Furthermore, the solutions satisfy a uniform large deviations principle on subsets of $(L^1_\gamma\cap \Ent_{\Phi,\gamma})(\R^d)$ with uniformly bounded entropy with respect to weakly $L^1(\R^d)$-compact subsets.
\end{thm}

\begin{proof}  The large deviations principle is a consequence of the weak approach to large deviations \cite[Theorem~6]{BudDupMar2008}, the equivalence of uniform Laplace and large deviations principles with respect to compact subsets of the initial data \cite[Theorem~4.3]{BudDupSal}, and the methods of \cite[Theorem~29]{FehGes19} which rely on Theorem~\ref{thm_unique}, Proposition~\ref{prop_entropy}, Proposition~\ref{prop_measure}, Theorem~\ref{thm_exist}, Proposition~\ref{prop_control}, and Theorem~\ref{thm_skel} above.\end{proof}

\section*{Acknowledgements}  The first author acknowledges support from the National Science Foundation DMS-Probability Standard Grant 2348650, the Simons Foundation Travel Grant MPS-TSM-00007753, and the Louisiana Board of Regents RCS Grant 20130014386.  The second author acknowledges support by the Max Planck Society through the Research Group ``Stochastic Analysis in the Sciences.'' This work was funded by the European Union (ERC, FluCo, grant agreement No. 101088488). Views and opinions expressed are however those of the author(s) only and do not necessarily reflect those of the European Union or of the European Research Council. Neither the European Union nor the granting authority can be held responsible for them.

\bibliography{WhiteNoise}
\bibliographystyle{plain}

\end{document}